\documentclass[a4paper,9pt]{amsart}

\usepackage{amssymb,statmath}
\usepackage{color}
\usepackage{mathrsfs}
\usepackage{enumerate}

\newtheorem{theorem}{Theorem}

\newtheorem{proposition}{Proposition}[section]
\newtheorem{lemma}{Lemma}[section]
\newtheorem{remark}{Remark}[section]
\newtheorem{example}{Example}[section]
\newtheorem{notation}{Notation}
\newtheorem{assumption}{Assumption}

\newtheorem{thm}{Theorem}

\newcommand{\R}{{\mathbb{R}}}
\newcommand{\Z}{{\mathbb{Z}}}
\newcommand{\N}{{\mathbb{N}}}

\newcommand{\<}{\langle}
\renewcommand{\>}{\rangle}
\newcommand{\Id}{{\mathrm{I}}}
\newcommand{\xii}{{\mathrm{|\xi|}}}
\newcommand{\supp}{{\mathrm{\,supp\,}}}

\renewcommand{\epsilon}{\varepsilon}
\newcommand{\e}{\epsilon}
\newcommand{\BMO}{{\mathrm{\,BMO\,}}}

\newcommand{\BQ}{{\mathrm{BQ\,}}}
\newcommand{\IBQ}{{\mathrm{IBQ\,}}}
\newcommand{\VBQ}{{\mathrm{VBQ\,}}}
\newcommand{\VVBQ}{{\mathrm{VVBQ\,}}}

\newcommand{\KG}{{\mathrm{KG\,}}}

\newcommand{\ROT}{{\mathrm{ROT\,}}}

\newcommand{\low}{{\mathrm{low\,}}}
\newcommand{\high}{{\mathrm{high\,}}}

\newcommand{\lin}{{\mathrm{lin\,}}}

\title[Dispersive-Dissipative estimates for Boussinesq and other equations]{Dispersive-Dissipative estimates for Boussinesq \\  and other generalized wave equations}

\author[M. D'Abbicco, M.R. Ebert, A. Lagioia]{Marcello D'Abbicco, Marcelo Rempel Ebert, Antonio Lagioia}

\thanks{Orcid: $0000-0003-1369-1157$ (M. D'Abbicco), 0000-0003-4592-252X (M.R. Ebert), 0009-0004-0641-974X (A. Lagioia)\\
The first and the third authors have been supported by PRIN 2022 ``Anomalies in partial differential equations and applications'' CUP H53C24000820006, and have been supported by the INdAM-GNAMPA Project CUP E53C23001670001. The second author has been visiting professor at University of Bari Aldo Moro in 2024 and is partially supported by São Paulo Research Foundation (FAPESP), grant number
2024/12753-8. The third author is supported by INdAM-GNAMPA Project CUP E5324001950001.}

\address{Marcello D'Abbicco, Antonio Lagioia, Department of Mathematics, University of Bari, Via E. Orabona 4, 70125 BARI - ITALY}

\address{Marcelo Rempel Ebert, Departamento de Computa\c{c}\~ao e Matem\'atica, Universidade de S\~ao Paulo, Av. Bandeirantes 3900, Ribeir\~ao Preto, SP, 14040-901, Brasil \& Universit\`a degli Studi di Bari Aldo Moro}

\email{marcello.dabbicco@uniba.it, ebert@usp.br, antonio.lagioia@uniba.it}

\begin{document}

\begin{abstract}
We derive dispersive-dissipative estimates for multipliers associated to a general class of wave-type equations, in particular for the viscous Boussinesq equation. The dispersion is related to the geometric hypotheses on the phase function and on the degeneracies that may happen at low and high frequencies. The dissipation interacts with the dispersion, influencing the decay rate of the solution.

\end{abstract}

\maketitle

\section{Introduction}\label{sec:intro}

In this paper, we obtain dispersive-dissipative estimates for multipliers in the form
\begin{equation}\label{eq:mult}
e^{it\omega(\rho(\xi))-t\alpha(\rho^2(\xi))},
\end{equation}
where $\rho:\R^n\setminus\{0\}\to\R_+$ is a positively homogeneous smooth function of degree~$1$, $\omega:\R_+\to\R$ and $\alpha:\R_+\to\R_+$. The term $\mathscr{F}^{-1}(e^{it\omega(\rho(\xi))})$ is an oscillatory integral, where $\mathscr{F}$ denotes the Fourier transform in the tempered distributions space $\mathcal S'$. This term produces the dispersive effect under suitable assumptions on $\omega$ and $\rho$, while the term $\mathscr{F}^{-1}(e^{-t\alpha(\rho(\xi)^2)})$ is responsible for dissipation and possibly smoothing. 
The main application of our analysis is for multipliers as in~\eqref{eq:mult} that originate from the study of the full symbol of suitable partial differential equations of evolutive type.

We are interested in applying our result to the case of weakly degenerate phase functions, namely, when $\rho\omega''(\rho)=\textit{o}(\omega'(\rho))$ as $\rho\to0$ or as $\rho\to\infty$. This degeneracy has the effect to weaken the dispersive character of the multiplier, still the dispersion is stronger than it is in the pure wave case, since in this latter scenario $\omega(\rho)=\rho$ so that $\omega''$ vanishes identically. The geometric properties of the hypersurface $\Sigma=\{\xi\in\R^n: \ \rho(\xi)=1\}$ also come into play in determining the dispersive effect. In particular, following as in~\cite{CMY}, we assume that $\Sigma$ is a convex hypersurface of finite type (see later, \textsection\ref{sec:finite}, see also~\cite{BNW}); the type order of $\Sigma$ influences the dispersive decay rate (see, for instance, \cite{Ru2009, Ru2012, Su1994, ZYF2004}).

The interplay between dispersion and dissipation creates new interesting structures of decay rate as $t\to\infty$ and/or it modifies the regularity of the solution. The viscous Boussinesq equation is of particular interest in this class.

We are interested in deriving $L^p-L^q$ estimates for the solution to the initial value problems associated to the models considered, with $1\leq p\leq 2\leq q\leq\infty$. We focus on sharp $L^p-L^q$ estimates for the dispersive-dissipative model in the endpoint case $p=1$ (and, by duality, $q=\infty$). To obtain these estimates, several tools of classic Fourier analysis fail to produce optimal estimates. Our approach is based on two steps. We first derive estimates for the purely dispersive model with initial data in the real Hardy space $H^1$ (see~\cite{FS72}, see later, \textsection\ref{sec:notation}) in place of $L^1$, then we use the property that $\mathscr{F}^{-1}(\alpha^\Xi\,e^{-t\alpha(\rho(\xi)^2)})$ maps $L^1$ into $H^1$ for positive $\Xi$, under suitable assumptions (see~\cite[Corollary 3.1]{DAE25}).

The $L^p-L^q$ estimates that we obtain can be applied to study several types of nonlinear problems. We provide an example of this application for the nonlinear (inviscid or viscous) Boussinesq equation, and possibly for other models. 

\subsection{The operator $A$ and the class of equations}

In the following, $A$ denotes a second-order homogeneous operator whose action is defined by
\begin{equation}\label{eq:operatorgen}
Af = \mathscr{F}^{-1} (\rho(\xi)^2\,\hat f),
\end{equation}
for $f\in\mathcal S$ and extended by density. Here $\hat f=\mathscr{F}f$.  We assume that the Hessian matrix $H_{\rho(\xi)}$ is positive semidefinite for any~$\xi\neq0$, and that the hypersurface
\[ \Sigma=\{\xi\in\R^n:\ \rho(\xi)=1\},  \]
is of finite type (see later, \textsection\ref{sec:finite}). The class of operators in~\eqref{eq:operatorgen} includes differential operators with real, constant, coefficients
\begin{equation}\label{eq:operator}
A = - \sum_{j,k=1}^n a_{jk} \partial_{x_j}\partial_{x_k},\qquad \text{with}\quad \rho(\xi)^2 = \sum_{j,k=1}^n a_{jk} \xi_j\xi_k>0, \quad \text{for any $\xi\neq0$.}
\end{equation}
%
For those operators, the hypersurface $\Sigma$ is of finite order $2$. In particular, $A=-c^2\Delta$ with $c>0$, corresponds to $\rho=c\xii$, the radial case. The class in~\eqref{eq:operatorgen} also includes operators $A$ that are not differential operators, but such that $A^\ell$ is a differential operator for some $\ell\geq2$, in the sense that
\begin{equation}\label{eq:Aell}
A^\ell= (-1)^\ell \sum_{|\gamma|=2\ell} c_\gamma \partial_x^\gamma,\quad \text{with}\qquad \rho(\xi)^{2\ell} =\sum_{|\gamma|=2\ell} c_\gamma \xi^\gamma>0,\quad \text{for any $\xi\neq0$.}
\end{equation}
See, for instance, Examples~\ref{ex:CMY}, \ref{ex:DFZ}, \ref{ex:ZYF}. For those operators, the hypersurface $\Sigma$ is of finite order at most $2\ell$.

Let $A$ be as in~\eqref{eq:operatorgen}. We consider equations in the form
\begin{equation}\label{eq:abgeneral}
u_{tt} + 2\alpha(A)\,u_t + \beta(A) u =0.
\end{equation}
In most cases, $\alpha$ and $\beta$ are polynomials. We perform the Fourier transform~$\mathscr{F}$ with respect to the space variable, letting $\hat u(t,\cdot)=\mathscr{F}u(t,\cdot)$. Then~\eqref{eq:abgeneral} reduces to the ordinary differential equation
\begin{equation}\label{eq:harmonic}
\hat u_{tt} + 2\alpha(\rho^2) \hat u_t + \beta (\rho^2)\,\hat u=0,
\end{equation}
whose characteristic equation $\lambda^2 + 2\lambda\alpha(\rho^2) + \beta(\rho^2)=0$ has roots
\[ \lambda_\pm=-\alpha(\rho^2)\pm i\omega(\rho), \qquad \omega(\rho)=\sqrt{\beta(\rho^2)-\alpha^2(\rho^2)},\]
provided that $\alpha^2<\beta$. This case corresponds to the so-called ``damped oscillations'' regime, using the language of the damped harmonic oscillator. Information on the fundamental solution of the equation may then be obtained studying the properties of the multipliers $e^{t\lambda_\pm}$, that is, multipliers as in~\eqref{eq:mult}, possibly replacing $\omega$ by $-\omega$.

The region where $\alpha^2(\rho^2)\geq\beta(\rho^2)$ is called overdamping region, since oscillations are canceled in the solution to~\eqref{eq:harmonic}. In this region, there is no dispersion, so the study of $L^p-L^q$ estimates is much simpler. We address the reader to the case of the wave equation with classical damping~\cite{MN03,N03} and with viscoelastic damping~\cite{P85,S00}, where $A=-\Delta$, $\beta(A)=A$, and $\alpha(A)=1$ or $\alpha(A)=A$, respectively.

\subsection{The application to weakly degenerate dissipative models}

As an application of our general results, we study different equations, focusing on models with the aforementioned degeneracies. The equations considered in our paper may be interpreted as variants of two basic models, the classic wave equation:
\begin{equation}\label{eq:W}\tag{W-eq}
u_{tt} + Au =0,
\end{equation}
and the classic plate equation:
\begin{equation}\label{eq:PL}
u_{tt} + A^2u =0,
\end{equation}
In the comparison of those two models, a strong degeneracy appears for~\eqref{eq:W}, since $\omega=\rho$; hence, $\omega''$ identically vanishes. On the other hand, $\omega=\rho^2$ for~\eqref{eq:PL} is not degenerate, due to $\omega''=2$. This difference influences the dispersive characters of the two models, even when the hypersurface $\Sigma$ is the same; as we will later see, the dispersive structure is related to the quantity $t^{-\frac{n-1}\tau}$ for~\eqref{eq:W} and to the quantity~$t^{-\frac{n-1}\tau-\frac12}$ for~\eqref{eq:PL}; here $\tau$ is the order of the hypersurface $\Sigma$ ($\tau=2$ if $A$ is a second order differential operator as in~\eqref{eq:operator}), see~\textsection\ref{sec:finite}.

The main object of our study is the Viscous Boussinesq model
\begin{equation}\label{eq:VBQ}\tag{BQ-eq}
u_{tt} + Au + A^2u + 2\e Au_t =0,
\end{equation}
with $\e\geq0$, sufficiently small. The case $\e=0$ corresponds to the inviscid model. The phase function $\omega$ associated to~\eqref{eq:VBQ} is weakly degenerate at $\rho=0$, in the aforementioned sense: $\rho\omega''(\rho)=\textit{o}(\omega'(\rho))$ as $\rho\to0$ (see later, \textsection\ref{sec:BQ}).

The same structure of degeneracy also appears, for $\e>0$, at low frequencies only ($\rho(\xi)<\e^{-1}$), in the wave equation with viscoelastic dissipation:
\begin{equation}\label{eq:Wvisco}\tag{WV-eq}
u_{tt} + Au + 2\e Au_t =0.
\end{equation}
The presence of dissipation weakens the strong degeneracy associated to the phase function $\omega=\rho$ obtained for $\e=0$, i.e., for~\eqref{eq:W}, reducing it to a degeneracy analogous to the one appearing for~\eqref{eq:VBQ} at low frequencies.

Another model of interest is the Damped Klein-Gordon equation:
\begin{equation}\label{eq:DKG}\tag{KG-eq}
u_{tt}+Au+u+2\e u_t =0.
\end{equation}
Equation~\eqref{eq:DKG} with $\e>0$ is also known as telegraph model (when $n=1$). The phase function for this model is degenerate as $\rho\to\infty$ for $\e\geq0$, in the aforementioned sense: $\rho w''(\rho)=\textit{o}(w')$ as $\rho\to\infty$.

\subsection{The application to other models}

Our result also applies, in a simpler way, to the cases where there is no degeneracy, in some sense to models more closely related to the plate model~\eqref{eq:PL} than to the wave model~\eqref{eq:W}.

We may consider the dissipative plate equation,
\begin{equation}\label{eq:DPl}\tag{P-eq}
u_{tt}+A^2u+2\e A u_t =0,
\end{equation}
or to the plate equation with mass and different dissipations:
\begin{equation}\label{eq:DPlmass}\tag{PM-eq}
u_{tt}+A^2u+u+2\e A^\gamma u_t =0,\qquad \gamma=0,1.
\end{equation}
Models with fractional dissipation may also be studied, letting $\gamma\in(0,1)$ in~\eqref{eq:DPlmass} or modifying other models; we mention the wave equation with structural dissipation and the Klein-Gordon equation with structural dissipation.

Our result may also be applied to models with different order in time. For instance, if we consider Schr\"odinger type equation with potential:
\begin{equation}\label{eq:Sch}\tag{S-eq}
u_t - i\omega(A^{\frac12})u + \alpha(A)u=0,
\end{equation}
then the Fourier transform of its fundamental solution is $e^{it\omega(\rho)-t\alpha(\rho^2)}$, as in~\eqref{eq:mult}. For a study in the case $\alpha=0$ and with positive $w''$, we address the reader to~\cite{Ozawa2011}. We mention that the assumption $\omega''>0$ is not necessary in our manuscript (if $\omega''$ vanishes at some value $\rho>0$, then we can apply Proposition~\ref{prop:maininter} in~\textsection\ref{sec:inter}).

Our results may also be applied to models with higher order derivatives in time, once they are in normal forms, provided that one obtains suitable properties on the phase function appearing in the fundamental solution of the associated initial value problem.


\subsection{The dispersive estimates}

In~\cite{CMY}, the authors obtain several results for oscillatory integrals of type $\mathscr{F}^{-1} (e^{it\omega(\rho(\xi))})$. In particular, they obtain pointwise estimates for the multipliers localized via a dyadic partition in frequencies, then they apply those estimates to derive estimates for the solution to several equations in Besov spaces  $B^s_{p,r}-B^s_{q,r}$, where $1\leq p\leq q\leq\infty$ and $r\in[1,\infty]$. Due to the Besov embeddings (for $s=0$ and $r=2$) that hold for $p\in(1,2]$ and $q\in[2,\infty)$, the authors may derive $L^p-L^q$ estimates, with $1<p\leq 2\leq q<\infty$.

In our manuscript, we slightly modify the approach in~\cite{CMY}, so that we include the endpoints $p=1$ and $q=\infty$, which cannot be directly obtained by the Besov embeddings; therefore, we derive $L^p-L^q$ estimates, with $1\leq p\leq2\leq q\leq\infty$. In a limit case, that is crucial for the application of the estimates when there is a dissipation interacting with the dispersion, the $L^1$ space is replaced by the real Hardy space $H^1$, and $L^\infty$ is replaced by $\BMO$, the space of functions with bounded mean oscillations (see later, Theorem~\ref{thm:mainlowdisp}).

Then we study the interplay between the dispersive estimates for $e^{it\omega(\rho(\xi))}$ and the dissipative estimates for $e^{-t\alpha(\rho^2(\xi))}$ in~\eqref{eq:mult}.





\subsection{Notation}\label{sec:notation}

In the paper, $\R_+=\{r\in\R: \ r>0\}$, $\mathcal S=\mathcal S(\R^n)$ is the Schwartz space of rapidly decreasing functions, and $\mathcal S'$, the dual space of $\mathcal S$, is identified with the space of tempered distributions.
\begin{notation}
We use the following notation for the minimum of two functions or quantities:
\[ f \wedge g = \min\{f,g\} \]
We also use the notation $f\lesssim g$ when $f,g\geq0$ and there exists $C>0$ (where the dependence of $C$ on parameters may be deduced by the context) such that $f\leq Cg$. We also use the notation $f\sim g$ when $f\lesssim g\lesssim f$. We use the notation $f\ll g$ or $g\gg f$ when we assume that $f\leq \varepsilon g$ holds for a sufficiently small~$\varepsilon>0$.
\end{notation}
\begin{notation}
Let $\gamma=(\gamma_1,\ldots,\gamma_n)\in\N^n$. We say that $|\gamma|=\gamma_1+\ldots+\gamma_n$ is the length of the multi-index $\gamma$ and we put
\[ \partial_x^\gamma=\partial_{x_1}^{\gamma_1}\ldots\partial_{x_n}^{\gamma_n},\qquad \xi^\gamma=\xi_1^{\gamma_1}\cdot\ldots\cdot\xi_n^{\gamma_n}. \]
\end{notation}
\begin{notation}
We use the notation $\<\rho\>=\sqrt{1+\rho^2}$. For any $\nu\in\R$, we use the notation
\[ |D|^\nu f=\mathscr{F}^{-1}(\xii^\nu\,\hat f), \qquad \<D\>^\nu f=\mathscr{F}^{-1}(\<\xi\>^\nu\,\hat f), \]
where $\mathscr{F}$ denotes the Fourier transform on $\R^n$, for $f$ in suitable functional or distributional space.
\end{notation}
\begin{notation}
We consider the usual Lebesgue spaces on $\R^n$:
\[ \begin{split}
L^q
    & =\{f \, \text{meas. on $\R^n$:} \ \int_{\R^n} |f(x)|^q\,dx<\infty\},\quad q\in[1,\infty),\\
L^\infty
    & =\{f \, \text{meas. on $\R^n$:} \ \exists L, \ m(\{x\in\R^n: \ |f(x)|>L\})=0\},
\end{split} \]
with their norms.
\end{notation}
\begin{notation}
We consider $M_p^q$ multipliers, in the sense that $\Phi\in M_p^q$ and
\begin{equation}
\label{eq:multiplier}
\|\Phi\|_{M_p^q} = \|\mathscr{F}^{-1}(\Phi)\ast \cdot \|_{L^p\to L^q},
\end{equation}
when $\Phi\in\mathcal S'(\R^n)$ and $\mathscr{F}^{-1}(\Phi)\ast$ may be extended to a linear bounded operator from $L^p$ to $L^q$.
\end{notation}
We recall that, in particular: $M_2^2=L^\infty$, $M_1^2=M_2^\infty=L^2$, $M_1^1=\mathscr{F}(\mathcal M)$, where $\mathcal M$ is the space of bounded measures, and $\|\Phi\|_{M_1^\infty}=\|\mathscr{F}^{-1}(\Phi)\|_{L^\infty}$ (\cite[Theorems 1.4, 1.5]{Ho}). In particular, $M_1^\infty$ contains distributions of positive order (\cite[Theorem 1.9]{Ho}), though in this manuscript we will mainly consider multipliers in $\mathcal S$, due to a dyadic localization.
\begin{notation}
We also define $M(H^1,L^q)$ and $M(H^1,H^1)$ as the space of bounded multipliers from the real Hardy space $H^1$ to $L^q$ and to itself, respectively. It is convenient for us to describe the real Hardy space $H^1$ as the subset of $L^1$ functions~$f$ such that the Riesz transforms of $f$ are also in $L^1$; we put
\begin{equation}\label{eq:H1norm}
\|f\|_{H^1} = \|f\|_{L^1} + \sum_{j=1}^n \|R_jf\|_{L^1},
\end{equation}
where $R_j$ are the Riesz transforms of $f$ defined by $R_jf=\mathscr{F}^{-1}(i(\xi_j/\xii)\,\hat f)$ (see~\cite{F71}). Therefore,
\begin{equation}\label{eq:H1multnorms}\begin{split}
\|m\|_{M(H^1,L^q)}&=\sup\big\{\|\mathscr{F}^{-1}(m\mathscr{F}(f))\|_{L^q}:f\in \mathcal{S}\cap H^1, \|f\|_{H^1}=1\big\}, \\
\|m\|_{M(H^1,H^1)}&=\sup\big\{\|\mathscr{F}^{-1}(m\mathscr{F}(f))\|_{H^1}:f\in \mathcal{S}\cap H^1, \|f\|_{H^1}=1\big\}.
\end{split}\end{equation}
Similarly, we define $M(L^p,\BMO)$, $M(H^1,\BMO)$ and $M(\BMO,\BMO)$, where $\BMO$ is the space of functions with bounded mean oscillations, and it is the dual of $H^1$, see~\cite[Theorem 2]{FS72}.
\end{notation}
\begin{notation}
For $1\leq p\leq 2\leq q\leq\infty$, we use the following notation to describe the influence of the dispersion on $L^p-L^q$ estimates:
\begin{equation}\label{eq:d}
d=d(p,q)=\min \left\{ \frac2p-1, 1-\frac2q \right\}.
\end{equation}
\end{notation}
In particular, $d(p,q)\leq1/p-1/q$, with the equality achieved if, and only if, $q=p/(p-1)$, that is, $(p,q)$ lies on the ``conjugate line''.

\subsection{Multiplier theorems in $L^p$ spaces and in the real Hardy space $H^1$}\label{sec:multipliers}

Here we collect a simple version of the Mikhlin-H\"ormander multiplier theorem, which we will use in the following.
\begin{thm}[Mikhlin-H\"ormander multiplier theorem](see~\cite[Theorem H]{SteinWeiss1960} and~\cite[Theorem E]{Miyachising})
\label{thm:Mik}
If $m\in\mathcal C^{[n/2]+1}(\R^n\setminus\{0\})$ and $|\partial_\xi^\alpha m(\xi)|\leq C\xii^{-|\alpha|}$ for $|\alpha|\leq [n/2]+1$, then $m\in M_p^p$ for any $p\in(1,\infty)$ and $m\in M(H^1,H^1)$.
\end{thm}
We also recall the Hardy-Littlewood-Sobolev inequality for fractional integration.
\begin{thm}[Hardy-Littlewood-Sobolev inequality](see~\cite[Theorem F]{Miyachising})
\label{thm:HLS}
Let $1<p<q<\infty$. Then $\xii^{-n\left(\frac1p-\frac1q\right)}\in M_p^q$. If $p=1$, then $\xii^{-n\left(1-\frac1q\right)}\in M(H^1,L^q)$.
\end{thm}

\subsection{Plan of the paper}

The plan of the paper is the following:
\begin{itemize}
\item in~\textsection\ref{sec:finite}, we provide the definition of hypersurface of finite order, with some remarks and examples;
\item in~\textsection\ref{sec:stationary}, we derive stationary phase lemmas, localized at low and high frequencies to treat multipliers of type $e^{it\omega(\rho(\xi))}$, extending the result in~\cite{CMY}; our result is of interest even if we only consider the easiest, radial case~$\rho(\xi)=\xii$, that is, $A=-\Delta$;
\item in~\textsection\ref{sec:estimates}, we apply the stationary phase lemmas, interpolating them with $L^2-L^2$ and $L^1-L^2$ estimates, at dyadic level, to derive $L^p-L^q$ estimates, with $1\leq p \leq2\leq q\leq \infty$; 
we carry on this analysis distinguishing the zone of low frequencies (Theorem~\ref{thm:mainlow}) and the zone of high frequencies (Theorems~\ref{thm:mainhi} and~\ref{thm:mainhiXi});
\item in~\textsection\ref{sec:models}, we show how our estimates apply to the models of interest;
\item in~\textsection\ref{sec:nonlinear}, we apply our estimates to the nonlinear Boussinesq equation and possibly to other nonlinear models;
\item in~\textsection\ref{sec:theoremprooflow} we give the proof of Theorem~\ref{thm:mainlow};
\item in~\textsection\ref{sec:theoremproofhi} we give the proof of Theorems~\ref{thm:mainhi} and~\ref{thm:mainhiXi};
\item in~\textsection\ref{sec:bneg}, we discuss models with $\rho\omega'(\rho)\to0$ as $\rho\to\infty$, as the Improved Boussinesq Equation;
\item in~\textsection\ref{sec:inter}, we discuss how to deal with models with degeneracies at some point $\rho\in(0,\infty)$, as happens for the plate equation with rotational inertia.
\end{itemize}

%


\section{Hypersurfaces of finite type}\label{sec:finite}

The hypersurface
\[ \Sigma = \{\xi\in\R^n: \ \rho(\xi)=1\}, \]
is smooth and compact as a consequence of the assumption that $\rho:\R^n\setminus\{0\}\to\R_+$ is a positively homogeneous smooth function of degree $1$. Moreover, $\Sigma$ is convex as a consequence of the assumption that $H_{\rho(\xi)}$ is positive semidefinite (the two conditions are equivalent, see~\cite[Proposition 2.1]{DY09}).

Then $\Sigma$ is of finite type $\tau\geq2$ (see~\cite[p.351]{BNW}) if $\tau$ is the least integer such that
\begin{equation}\label{eq:taucond1}
\sum_{j=1}^\tau |\<\eta,\nabla\>^j\rho(\xi)|\geq C_\tau, \quad \xi\in\Sigma,\ \eta\in S^{n-1},
\end{equation}
where $S^{n-1}=\{\eta: \ |\eta|=1\}$ is the unit sphere. By~\cite[Proposition 4.3]{CMY}, for any $\lambda>0$, the above condition is equivalent to
\begin{equation}\label{eq:taucond}
\sum_{j=1}^\tau |\<\eta,\nabla\>^j(\rho(\xi)^\lambda)|\geq C_\tau, \quad \xi\in\Sigma,\ \eta\in S^{n-1}.
\end{equation}
%
By~\cite[Proposition 4.2]{CMY} (see also~\cite[Proposition 2.2]{DY09}), 
$\tau=2$ if, and only if, one of the following is true:
\begin{itemize}
\item the rank of the Hessian of $\rho$ is maximal for any $\xi\neq0$, i.e., it is $n-1$ (since $\rho$ is homogeneous of degree $1$, the rank is at most $n-1$);
\item the Hessian of $\rho^\lambda$ is nonsingular, i.e., $\det H_{\rho^\lambda}\neq0$ for any $\xi\neq0$, for some $\lambda>1$;
\item the Gaussian curvature of $\Sigma$ is nonzero everywhere.
\end{itemize}
In the following, we assume that $\Sigma$ is of finite type.

We stress that if $A$ is as in~\eqref{eq:Aell}, then $\tau\leq 2\ell$ (see \cite[Proposition 2.2]{ZYF2004}). In particular, $\tau=2$ if $A$ is as in~\eqref{eq:operator}. On the other hand, $\tau>2$ if, and only if, the determinant of the Hessian of $\rho(\xi)^{2\ell}$ vanishes at some $\xi\neq0$.

In many equations of physical interest, $A=-\Delta$, then $\rho=\xii$ and~$\tau=2$, but other scenarios may occur. 
The forthcoming Examples~\ref{ex:CMY}, \ref{ex:DFZ} and~\ref{ex:ZYF} are contained in~\cite{ZYF2004}.
\begin{example}\label{ex:CMY}
An example of interest with $\tau=2\ell$, $\ell\geq2$, is the model
\[ u_{tt} + (-1)^\ell \sum_{j=1}^n \partial_{x_j}^{2\ell} u = u_{tt}+A^\ell u, \quad \text{where}\qquad \rho=\sqrt[2\ell]{\xi_1^{2\ell}+\ldots+\xi_n^{2\ell}}. \]
%
The case $\ell=2$ appears in some nonlinear Boussinesq equations~\cite{Sh}.
\end{example}
\begin{example}\label{ex:DFZ}
An example with $\tau=4$ is the following model in $\R^2$: 
\[ u_{tt} +\partial_{x_1}^4u + 6\partial_{x_1}^2\partial_{x_2}^2u + \partial_{x_2}^4u = u_{tt}+ A^2 u, \quad \text{where}\qquad  \rho=\sqrt[4]{\xi_1^4 + 6\xi_1^2\xi_2^2 + \xi_2^4}. \]
To check that $\tau=4$ it is convenient to study~\eqref{eq:taucond} with $\lambda=4$; in particular, to show that $\tau=4$ it is sufficient to fix $\xi_2=-\xi_1$ and $\eta_1=\eta_2$ to get that
\[ \sum_{j=1}^3 |\<\eta,\nabla\>^j(\rho(\xi)^4)|=0. \]
\end{example}
\begin{example}\label{ex:ZYF}
Another example with $\tau=4$ is the following model in $\R^2$:
\[ u_{tt} -\partial_{x_1}^6u - 5\partial_{x_1}^4\partial_{x_2}^2u - \partial_{x_2}^6u = u_{tt}+ A^3 u, \quad \text{where}\qquad \rho=\sqrt[6]{\xi_1^6 + 5\xi_1^4\xi_2^2 + \xi_2^6}. \]
\end{example}
%
%
%
%
%

\section{Stationary phase lemmas}\label{sec:stationary}

We employ a dyadic partition in annuli to study the phase function. For any $N\in 2^\Z$ (that is, $N=2^j$, with $j\in\Z$), we define
\[ A_N = \{ \xi\neq0: \ 2^{-1}N \leq \xii < 2N \}, \]
and we fix $\psi_N(\xi)=\psi(\xi/N)$, where $\psi\in\mathcal C_c^\infty$ with $\supp\psi\subset A_1$, such that $\sum_{N\in 2^\Z} \psi_N =1$. A method to construct such $\psi$, nonnegative, is to fix $\chi\in\mathcal C_c^\infty$ with $\chi=0$ for $\xii\geq1$ and $\chi=1$ for $\xii\leq1/2$, then define $\psi(\xi)=\chi(\xi/2)-\chi(\xi)$. Let
\begin{equation}\label{eq:IN}
I_N=N^{-n-\mu}\,\int_{\R^n} e^{ix\xi-it\omega(\rho(\xi))}\,(\rho(\xi))^\mu\,\psi_N(\rho(\xi))\,d\xi,
\end{equation}
for some $\mu\in\R$. The methodology in~\cite{CMY} (in particular, see~\cite[Remark 1.2]{CMY}) is based on a change of variable and on~\cite[Theorem~B]{BNW} to reduce to the study to a linear oscillating integral:
\[ \begin{split}
I_N
    & =\int_0^\infty e^{-it\omega(Nr)}\,r^{\mu+n-1}\psi(r)\,\int_\Sigma e^{iN|x|r\<x/|x|,\eta\>} \frac{d\sigma(\eta)}{|\nabla \rho(\eta)|}\, dr\\
    & =\sum_\pm \int_0^\infty e^{irN|x|\<x/|x|,\eta_\pm\>-it\omega(Nr)}\,H_\pm(N|x|r)\,r^{\mu+n-1}\psi(r)\, dr,
\end{split}\]
where the unit outer normal vector at $\eta_\pm \in \Sigma$ is $\pm x/|x|$, and
\[ |H_\pm^{(j)}(N|x|r)|\leq C_j\,(1+N|x|r)^{-\frac{n-1}\tau-j}. \]
\begin{remark}\label{rem:Bessel}
If $\Sigma$ is the unit sphere $S^{n-1}=\{\eta\in\R^n: \ |\eta|=1\}$, then a simpler, direct, proof may be given using Bessel functions. Indeed, $\rho(\xi)=\xii$ so that $|\nabla \rho(\eta)|=1$; hence,
\[ \begin{split}
\int_{S^{n-1}} e^{iN|x|r\<x/|x|,\eta\>} \frac{d\sigma(\eta)}{|\nabla \rho(\eta)|}
    & = |S^{n-1}|\,\< \delta_1, e^{iNr\<x,\eta\>} \> = |S^{n-1}|\,\mathscr{F}(\delta_1)(-Nxr) \\
    & = |S^{n-1}|\,|Nxr|^{1-\frac{n}2}\,\mathcal J_{\frac{n}2-1} (N|x|r),
\end{split} \]
where $\delta_1$ is the unit measure on the unit sphere $S^{n-1}$, and the Bessel function $\mathcal J_\gamma(r)$ has the asymptotic expansion for $\gamma\geq0$ and $N|x|\to\infty$,
\[ \mathcal J_{\frac{n}2-1} (N|x|r) = (2/\pi)^{\frac12}\, (N|x|r)^{-\frac12} \cos (N|x|r- (n/2-1)\pi/2-\pi/4) + \textit{O}((N|x|r)^{-\frac32}). \]
Therefore, letting $\eta_\pm = \pm x/|x|$, we may explicitly compute $H_\pm(N|x|r)$; we stress that $\tau=2$.
\end{remark}
We have the following.
\begin{lemma}\label{lem:low}
Let $N_0$ be such that for any $\rho\leq 2N_0$ it holds:
\begin{align}
\label{eq:wprimelow}
& \rho|\omega'(\rho)|\sim \rho^a,\quad \text{for some $a\in\R$;}\\
\label{eq:wboundlow}
&\rho^j|\omega^{(j)}(\rho)|\lesssim \rho^a,\qquad j\geq2;\\
\label{eq:derlow}
&\rho^m|\omega^{(m)}(\rho)|\gtrsim \rho^{\tilde{a}}, \quad \text{for some $m\geq2$ and $\tilde a\geq a$.}
\end{align}
Then, for all $N\leq N_0$ we may estimate
\begin{equation}\label{eq:INgood}
|I_N|\leq C_L\,(1+tN^a)^{-L},
\end{equation}
for any $L>0$, exception given for a finite number of $N$, for which the following estimate holds:
\begin{equation}\label{eq:INlow}
|I_N|\lesssim 1\wedge (tN^a)^{-\frac{n-1}\tau} \left(1\wedge (t\,N^{\tilde{a}})^{-\frac1m} \right).
\end{equation}
If~\eqref{eq:derlow} is removed from the assumptions, the above estimate holds formally setting $m=\infty$, that is, $|I_N|\lesssim 1\wedge (tN^a)^{-\frac{n-1}\tau}$, for any $N\leq N_0$.
\end{lemma}
In the nondegenerate case $\tilde a=a$, \eqref{eq:INlow} reduces to
\begin{equation}\label{eq:INlownodeg}
|I_N|\lesssim 1\wedge (tN^a)^{-\frac{n-1}\tau-\frac1m}.
\end{equation}
\begin{remark}\label{rem:adeg}
It is expected that the degeneracy $\tilde a>a$ appears when $a=1$, for smooth $\omega$. Indeed, if $\omega$ is smooth in $\rho=0$, by Taylor's expansion, $\omega'(\rho)\sim\rho^{a-1}$, so that $\rho^k\omega^{(k)} = \textit{o}(\rho^a)$ for $k\geq2$, if $a=1$. This is the scenario occurring in~\textsection\ref{sec:BQ} and in~\textsection\ref{sec:visco} (see also Appendix~\ref{sec:bneg}).
\end{remark}
\begin{proof}
It is sufficient to apply~\cite[Theorem 1.1]{CMY}, after taking into consideration of the following property. Since we assumed that $\omega^{(m)}$ does not vanish for $\rho\in(0,2N_0)$, it follows that for any given $t>0$ and $x\in\R^n$, the phase function
\[ i\rho N|x|\,\<x/|x|,\eta_\pm\> -it\omega(N\rho),  \]
admits critical points for a finite number of values of $N$, corresponding to the solutions to
\[ \omega'(N\rho) = t^{-1}|x|\,\<x/|x|,\eta_\pm\>.  \]
Therefore, for all values of $N$, but a finite number of them, the estimate in Theorem 1.1 may be improved using the result for oscillatory integrals in absence of stationary points (see~\cite[Chapter 8, Proposition 2.1]{SS}). For the remaining values of $N$, estimate (1.14) in~\cite[Theorem 1.1]{CMY} with $\theta=0$ and with $\theta=1$ gives, respectively, $|I_N|\lesssim 1$ and $|I_N|\lesssim (tN^a)^{-\frac{n-1}\tau}$. On the other hand, estimate (1.15) in~\cite[Theorem 1.1]{CMY} with $\theta=1$ gives $|I_N|\lesssim (tN^a)^{-\frac{n-1}\tau} t^{-\frac1m}\,N^{-\frac{\tilde{a}}m}$. This proves~\eqref{eq:INlow}. If we remove~\eqref{eq:derlow}, then we can only obtain the same result in~\cite[Theorem 1.1]{CMY}, that is, $|I_N|\lesssim 1\wedge (tN^a)^{-\frac{n-1}\tau}$ for any $N\leq 2N_0$.
\end{proof}
\begin{lemma}\label{lem:hi}
Let $N_1$ be such that for any $\rho\geq N_1/2$ it holds:
\begin{align}
\label{eq:wprimehi}
& \rho|\omega'(\rho)|\sim \rho^b,\quad \text{for some $b\in\R$;}\\
\label{eq:wboundhi}
&\rho^j|\omega^{(j)}(\rho)|\lesssim \rho^b,\qquad j\geq2;\\
\label{eq:derhi}
&\rho^m|\omega^{(m)}(\rho)|\gtrsim \rho^{\tilde{b}}, \quad \text{for some $m\geq2$ and $\tilde b\leq b$.}
\end{align}
Then, for all $N\geq N_1$ we may estimate~\eqref{eq:INgood} for any $L>0$, exception given for a finite number of $N$, for which the following estimate holds:
\begin{equation}\label{eq:INhi}
|I_N|\lesssim 1\wedge (tN^b)^{-\frac{n-1}\tau} \left(1\wedge (t\,N^{\tilde{b}})^{-\frac1m} \right).
\end{equation}
If~\eqref{eq:derhi} is removed from the assumptions, the above estimate holds formally setting $m=\infty$, that is, $|I_N|\lesssim 1\wedge (tN^b)^{-\frac{n-1}\tau}$, for any $N\geq N_1$.
\end{lemma}
In the nondegenerate case $\tilde b=b$, \eqref{eq:INhi} reduces to
\begin{equation}\label{eq:INhinodeg}
|I_N|\lesssim 1\wedge (tN^b)^{-\frac{n-1}\tau-\frac1m}.
\end{equation}
\begin{remark}\label{rem:bdeg}
Similarly to Remark~\ref{rem:adeg}, it is expected that the degeneracy $\tilde b<b$ appears when $b=1$ for ``smooth at infinity'' $\omega$, that is, for $\omega'(\rho)=f(1/\rho)$ with $f$ smooth at $r=0$. Indeed, by Taylor's expansion, $\omega'(\rho)\sim\rho^{b-1}$, so that $\rho^k\omega^{(k)} = \textit{o}(\rho^b)$ for $k\geq2$ if $b=1$. This is the scenario occurring in~\textsection\ref{sec:KG}.
\end{remark}
\begin{proof}
The proof is completely analogous to the proof of Lemma~\ref{lem:low}.
\end{proof}
%


\section{Estimates for evolution equations}\label{sec:estimates}

Here and in~\textsection\ref{sec:models}, we assume $A$ as in~\eqref{eq:operatorgen} and $\Sigma$ a convex hypersurface of finite type $\tau$, as in~\textsection\ref{sec:finite}. We are interested in applying Lemmas~\ref{lem:low} and \ref{lem:hi} to derive $L^p-L^q$ estimates, $1\leq p\leq 2\leq q\leq\infty$, for the solution to initial value problems for evolution equations in the form~\eqref{eq:abgeneral}
\begin{equation}\label{eq:CPgen}\begin{cases}
u_{tt} + 2\alpha(A)u_t + \beta(A)u =0, & t>0,\ x \in\R^n,\\
(u,u_t)(0,x)=(u_0,u_1)(x).
\end{cases} \end{equation}
For any $\xi\in\R^n$ such that $\alpha^2(\rho(\xi)^2)<\beta(\rho(\xi)^2)$, we get
\begin{equation}\label{eq:fundamental}
\hat u(t,\xi) = e^{-\alpha(\rho(\xi)^2)t}\, \left(\frac{\sin(t \omega(\rho(\xi)))}{\omega(\rho(\xi))}\,\hat u_1(\xi) + \cos(t\omega(\rho(\xi)))\,\hat u_0(\xi)\right)\,.
\end{equation}
By $L^p-L^q$ estimates we mean estimates of type
\[ \||D|^\nu\partial_t^\ell u(t,\cdot)\|_{L^q}\leq f(t)\,\|(u_0,u_1)\|_{L^p},\quad t>0, \]
where $\nu\in\R$ and $\ell\in\N$, for some $f(t)$, or variant of those estimates. Among those estimates, the so-called dispersive estimates, obtained when $p=q'=q/(q-1)$, are of particular interest to study nonlinear models, though $L^1-L^q$ estimates can also be important in some scenario (see, for instance, \cite{EL19}).

In view of
\[ \partial_t^\ell e^{\pm it\omega-t\alpha} = (\pm i\omega-\alpha)^\ell\,e^{\pm it\omega}\,e^{-t\alpha}, \]
we assume that
\begin{equation}\label{eq:thetalow}
\alpha(\rho^2)\sim \rho^{\theta}, \quad \rho^j|\alpha^{(j)}(\rho^2)|\lesssim \rho^{\theta}, \quad j\geq1,
\end{equation}
for some $\theta\geq0$.
\begin{remark}\label{rem:theta}
Due to the assumption $\alpha^2<\beta$, it follows that $\alpha^2 = \textit{O}(\omega^2)$. In particular, if $\omega(\rho)\sim\rho^a$ as $\rho\to0$, then $a\leq\theta$ and, if $\omega(\rho)\sim\rho^b$ as $\rho\to\infty$, then $\theta\leq b$.
\end{remark}
We aim to estimate in $M_p^q$ the multiplier (see~\eqref{eq:multiplier}) $g\,e^{\pm it\omega-\alpha t}$, where
\begin{equation}\label{eq:g}
g(\xi) = \xii^\nu\,(\pm i\omega-\alpha)^\ell\,,
\end{equation}
for some $\nu\in\R$ is the function that takes into account of $\nu$ derivative in space and $\ell$ derivatives in time.

To localize the multiplier $g\,e^{\pm it\omega-t\alpha}$ at low frequencies, or at high frequencies, we multiply it by
\begin{equation}\label{eq:chi}
\chi_0=\sum_{N\leq N_0} \psi_N(p), \qquad \chi_1=\sum_{N\geq N_1} \psi_N(p),
\end{equation}
respectively. In particular, $\chi_0+\chi_1=1$ if we set $N_1=2N_0$. We estimate $g\,e^{\pm it\omega-t\alpha}\psi_N$ using Lemmas~\ref{lem:low} for $N\leq N_0$, and Lemma~\ref{lem:hi} for $N\geq N_1$.

However, if we replace $ge^{\pm it\omega-t\alpha}$ by
\begin{equation}\label{eq:Phisinc}
g\,e^{-t\alpha}\,(e^{it\omega}-e^{-it\omega}),\quad \text{where}\quad e^{it\omega}-e^{-it\omega}= 2i\sin(t\omega),
\end{equation}
as it happens in~\eqref{eq:fundamental}, we may weaken some assumption in our results, thanks to the cancellations in~\eqref{eq:Phisinc} as $t\omega\to0$.

At low frequencies and at high frequencies, we plan to estimate the function $g$ by a suitable power $\rho^\mu$, in $M_p^p$, in the sense that $\|g\,\rho^{-\mu}\chi_0\|_{M_p^p}\leq C$ (this is trivial if $1<p<\infty$, since it suffices to use Mikhlin-H\"ormander theorem~\ref{thm:Mik} in~\textsection\ref{sec:multipliers}), and similarly when $\chi_0$ is replaced by $\chi_1$. In particular, $\mu=\nu$ when $\ell=0$, otherwise $\mu$ may also depend on $\alpha$ and $\theta$ as in~\eqref{eq:thetalow}.

To prove our estimates, we will employ a dyadic partition. In any dyadic set~$A_N$, we will interpolate $L^2-L^2$, $L^1-L^2$ and $L^1-L^\infty$ estimates. This interpolation (and the fact that the $L^2-L^\infty$ estimate is the dual of the $L^1-L^2$ estimate) allow us to derive $L^p-L^q$ estimates for $1\leq p\leq 2\leq q\leq\infty$. Similarly, it is possible to prove $B_{p,r}^s-B_{q,r}^s$ estimates in Besov spaces, for $1\leq p\leq q\leq\infty$, as in~\cite{CMY}. By the Besov embeddings, these estimates imply the $L^p-L^q$ estimates for $1<p\leq 2\leq q<\infty$, so our result is mostly new for $p=1$ and/or $q=\infty$.

In this manuscript, we do not consider $L^p-L^q$ estimates with $1\leq p\leq q<2$ or its dual range $2<p\leq q\leq\infty$, since we are mostly interested in analyzing the dispersive character of the equations. To deal with those cases, we shall modify the proof, possibly replacing the dyadic approach with an argument based on complex interpolation, as done in~\cite{DAE25} for homogeneous phase functions.
\begin{remark}\label{rem:m2}
In the application of the forthcoming Theorems~\ref{thm:mainlow} and~\ref{thm:mainhi} to the considered models, we will fix $m=2$. This because the use of conditions~\eqref{eq:derlow} and~\eqref{eq:derhi} for $m>2$ does not bring any benefit to the estimates in the case when $\tilde a(m)$ is a nondecreasing function with respect to $m\geq2$, and $\tilde b(m)$ is nonincreasing with respect to $m\geq2$ (a ``highly reasonable'' assumption, verified for all the models considered). Indeed, in this case,
\[\begin{split}
& t^{-\frac1m}N^{-\frac{\tilde a(m)}m}\leq t^{-\frac1m}N^{-\frac{\tilde a(2)}m} \leq 1 \wedge t^{-\frac12}N^{-\frac{\tilde a}2},\qquad \text{for $N\leq1$ and $m\geq3$,}\\
& t^{-\frac1m}N^{-\frac{\tilde b(m)}m} \leq t^{-\frac1m}N^{-\frac{\tilde b(2)}m} \leq 1 \wedge t^{-\frac12}N^{-\frac{\tilde b(2)}2}, \qquad \text{for $N\geq1$ and $m\geq3$.}
\end{split} \]
\end{remark}
\begin{remark}
Our result is also applicable to the initial value problem for Schr\"odinger type equations with potential:
\begin{equation}\label{eq:CPgenShr}\begin{cases}
u_t - i\beta(A)u + \alpha(A)u =0, & t>0,\ x \in\R^n,\\
u(0,x)=u_0(x),
\end{cases} \end{equation}
whose solution verifies, under the assumption $\beta(\rho^2)=\omega(\rho)>0$ and $\alpha(\rho^2)>0$,
\[ \hat u(t,\xi) = e^{i\omega(\rho(\xi))t-\alpha(\rho(\xi)^2)t}\,\hat u_0(\xi). \]
In this case, we may obtain estimates of type
\[ \||D|^\nu\partial_t^\ell u(t,\cdot)\|_{L^q}\leq f(t)\,\|u_0\|_{L^p},\quad t>0. \]
We stress that the restriction on $\theta$ appearing in Remark~\ref{rem:theta} does not appear for~\eqref{eq:CPgenShr}.
\end{remark}

\subsection{Low frequencies}

Localizing at low frequencies using $\chi_0$, we have the following result.
\begin{theorem}\label{thm:mainlow}
For any~$\xi$ with $\rho(\xi)\leq 2N_0$, we assume the following:
\begin{itemize}
\item that~\eqref{eq:wprimelow} and~\eqref{eq:wboundlow} hold for some $a>0$;
\item that~\eqref{eq:derlow} holds for some $m\geq2$ and $\tilde a\geq a$;
\item that~\eqref{eq:thetalow} holds for some $\theta>0$; if $\alpha=0$, we formally set $\theta=\infty$.
\end{itemize}
Let $1\leq p\leq 2\leq q\leq\infty$ and $d=d(p,q)$ as in~\eqref{eq:d}. We define
\begin{equation}\label{eq:kappa}
\kappa = \mu+n\left(\frac1p-\frac1q\right),
\end{equation}
with $\mu\in\R$ such that $\kappa>0$. Then,
\begin{equation}\label{eq:estlow}
\begin{split}
& \| \rho^\mu e^{it\omega(\rho)-t\alpha(\rho^2)}\chi_0(\rho) \|_{M_p^q} \leq C\,(1+t)^{-\delta_\low},\qquad t\geq0,\\
& \delta_\low = \begin{cases}
\frac{\kappa}\theta, & \theta\leq a,\\
\min\left\{\frac\kappa{a},\frac\kappa\theta+\left(1-\frac{a}\theta\right)\frac{n-1}\tau d\right\}, & a\leq\theta\leq\tilde a,\\
\min\left\{\frac\kappa{a},\frac\kappa{\tilde a}+\left(1-\frac{a}{\tilde a}\right)\frac{n-1}\tau d, \frac\kappa\theta + \left(1-\frac{a}\theta\right)\frac{n-1}\tau d + \left(1-\frac{\tilde a}\theta\right)\frac1m\,d \right\}, & \tilde a\leq\theta.
\end{cases}
\end{split}
\end{equation}
Estimate~\eqref{eq:estlow} also holds if $\kappa=0$ (so that $\delta_\low=0$) and $1<p\leq 2\leq q<\infty$.

\bigskip

If $w(\rho)\sim \rho^a$ as $\rho\to0$, then~\eqref{eq:estlow} remains valid replacing $e^{it\omega}$ by $e^{it\omega}-e^{-it\omega}$ and modifying $\delta_\low =\frac{\kappa+a}\theta-1$ if $\theta\leq a$. Also, we may weaken our assumption~$\kappa>0$ to $\kappa>-a$, or $\kappa\geq0$ to $\kappa\geq-a$ if $1<p\leq 2\leq q<\infty$.
\end{theorem}
We postpone the proof of Theorem~\ref{thm:mainlow} to~\textsection\ref{sec:theoremprooflow}.
\begin{remark}
We may include the case $a=1$ in Theorem~\ref{thm:mainlow}, which is excluded in~\cite[Theorem 1.2]{CMY}.
\end{remark}

We notice that in the pure nondegenerate case $\tilde a=a$ and $m=2$, $\delta_\low$ in~\eqref{eq:estlow} reduces to
%
\begin{equation}\label{eq:estlow0deg}
\delta_\low = \begin{cases}
\frac{\kappa}\theta, & \theta\leq a,\\
\min\left\{\frac\kappa{a}, \frac\kappa\theta + \left(1-\frac{a}\theta\right)\left(\frac{n-1}\tau+\frac12  \right)d \right\}, & a\leq\theta.
\end{cases}
\end{equation}
In particular, when $\tau=2$, we get the decay $t^{-\min\left\{\frac{\kappa}{a},\frac\kappa\theta + \left(1-\frac{a}\theta\right)_+\frac{n}2d\right\}}$ in~\eqref{eq:estlow0deg}.
\begin{remark}\label{rem:decaylow}
If $n=1$, then~\eqref{eq:estlow} reads as
\[ \| \rho^\mu e^{it\omega(\rho)-t\alpha(\rho^2)}\chi_0(\rho)\|_{M_p^q} \leq C\, (1+t)^{-\min\left\{\frac{\kappa}{\tilde a}, \frac\kappa\theta + \left(1-\frac{\tilde a}\theta\right)_+\frac1md\right\}},\qquad t\geq0. \]
On the other hand, if $n\geq2$, we stress that
\[ \frac\kappa{\tilde a}+\left(1-\frac{a}{\tilde a}\right)\frac{n-1}\tau d, \]
is a convex combination of the powers $\frac\kappa{a}$ and $\frac{n-1}\tau d$; the first power corresponds to a pure scaling decay, while the second one corresponds to a dispersive-type decay.
\end{remark}
%
To emphasize the influence of the dissipation, we distinguish three intervals for $\theta$.
\begin{itemize}
\item if $\theta\leq a$, then $\delta_\low=\kappa/\theta$ (the case $\theta<a$ is only relevant when $\rho\omega'(\rho)=\textit{o}(\omega)$, see Remark~\ref{rem:theta}, as in the forthcoming models in~\textsection\ref{sec:platemass} and in~ \textsection\ref{sec:KG}), or $\delta_\low = (\kappa+a)/\theta -1$ if $\omega(\rho)\sim\rho^a$ and we consider $g(e^{it\omega}-e^{-it\omega})$;
\item if $a\leq\theta\leq\tilde a$, then
\begin{equation}\label{eq:deltalowtheta2}
\delta_{\low} = \begin{cases}
\frac\kappa{a}, & \frac{\kappa}{a}\leq\frac{n-1}\tau d,\\
\frac\kappa\theta + \frac{n-1}\tau d\left(1-\frac{a}\theta\right), & \frac{\kappa}{a}\geq\frac{n-1}\tau d;
\end{cases} \end{equation}
\item if $\theta\geq\tilde a$, then
\begin{equation}\label{eq:deltalowtheta3}
\delta_\low = \begin{cases}
\frac\kappa{a}, & \frac{\kappa}{a}\leq\frac{n-1}\tau d,\\
\frac{\kappa}{\tilde a}+\left(1-\frac{a}{\tilde a}\right)\frac{n-1}\tau d, & 0\leq\frac{\kappa}{a}-\frac{n-1}\tau d\leq\frac1m\,\frac{\tilde a}{a} d,\\
\frac\kappa\theta + \left(1-\frac{a}\theta\right)\frac{n-1}\tau d + \left(1-\frac{\tilde a}\theta\right)\frac1md, & \frac{\kappa}{a}\geq\left(\frac{n-1}\tau +\frac1m\,\frac{\tilde a}{a}\right) d.
\end{cases} \end{equation}
In particular, when there is no damping, formally putting $\theta=\infty$, the last case in~\eqref{eq:deltalowtheta3} reads as
\begin{equation}\label{eq:deltaundamped}
\delta_\low = \left(\frac{n-1}\tau +\frac1m\right)d,\qquad \text{if}\quad \frac{\kappa}{a}\geq\left(\frac{n-1}\tau +\frac1m\,\frac{\tilde a}{a}\right) d.
\end{equation}
\end{itemize}
%

%
%

\subsection{High frequencies}

Localizing at high frequencies, we assume $b>0$. We postpone the discussion of the case $b<0$ that appears in some model as the Improved Boussinesq equation, to~\textsection\ref{sec:bneg}.

Moreover, we shall distinguish the cases of short time $t\in(0,1]$ and large time $t\geq1$.

For the ease of reading, we first consider the case with no damping $\alpha=0$.

According to Remark~\ref{rem:bdeg}, it is expected that a degeneracy appears if $b=1$, so that $\tilde b\leq0$. For this reason, for the sake of brevity, we assume either $\tilde b=b$ (no degeneracy) or $\tilde b<0<b$. The case $0\leq\tilde b<b$ may still be treated with minor modifications in the proof of Theorem~\ref{thm:mainhi}.
\begin{theorem}\label{thm:mainhi}
For any~$\xi$ with $\rho(\xi)\geq N_1/2$, we assume the following:
\begin{itemize}
\item that~\eqref{eq:wprimehi} and~\eqref{eq:wboundhi} hold for some $b>0$;
\item that~\eqref{eq:derhi} holds for some $m\geq2$ and either $\tilde b=b$ or $\tilde b<0$.
\end{itemize}
Let $1\leq p\leq 2\leq q\leq\infty$ and $d=d(p,q)$ as in~\eqref{eq:d}. We set~$\kappa$ as in~\eqref{eq:kappa} and we assume that $\mu\in\R$ is such that $\kappa>0$. If $\tilde b=b$, then
\begin{equation}\label{eq:esthi0deg}
\| \rho^\mu e^{it\omega(\rho)}\chi_1(\rho) \|_{M_p^q} \leq \begin{cases}
C\,t^{-\left(\frac{n-1}\tau + \frac1m\right)d}, & t\geq1, \\
C\,t^{-\frac{\kappa}{b}},& t\in(0,1],
\end{cases}
\end{equation}
provided that
\begin{equation}\label{eq:reg}
\frac\kappa{b} \leq \left(\frac{n-1}\tau + \frac1m\right) d.
\end{equation}
In the degenerate case $\tilde b<0<b$, then
\begin{equation}\label{eq:esthi}
\| \rho^\mu e^{it\omega(\rho)}\chi_1(\rho) \|_{M_p^q} \leq \begin{cases}
C\,t^{-\min\left\{\left(\frac{n-1}\tau+\frac1m\right)d, \frac\kappa{\tilde b}+\left(1+\frac{b}{-\tilde b}\right)\frac{n-1}\tau d\right\}}, & t\geq1, \\
C\,t^{-\frac{\kappa}{b}},& t\in(0,1],
\end{cases}
\end{equation}
provided that
\begin{equation}\label{eq:shrink}
\frac\kappa{b} \leq \frac{n-1}\tau d.
\end{equation}
Estimate~\eqref{eq:esthi0deg} and~\eqref{eq:esthi} also holds if $\kappa=0$ and $1<p\leq 2\leq q<\infty$, while if $\kappa<0$ they hold for $1\leq p\leq 2\leq q\leq\infty$, replacing $C\,t^{-\frac{\kappa}{b}}$ by $C$ for $t\in(0,1]$.
\end{theorem}
We postpone the proof of Theorem~\ref{thm:mainhi} to~\textsection\ref{sec:theoremproofhi}.

%
\begin{remark}
We stress that the possible degeneracy of $w''$ as $\rho\to\infty$, in the case $\tilde b<0<b$, impacts on both the decay rate and the region $(p,q)$ such that $\rho^\mu e^{it\omega(\rho)}\chi_1\in M_p^q$, namely, there is a loss of regularity, in view of~\eqref{eq:shrink}. This loss is consistent with the known results for the Klein-Gordon equation, for which the regularity is the same as per the massless wave equation, see later, \textsection\ref{sec:KG}, and~\cite[Theorem A]{MarshallStrauss1980}.
\end{remark}

\begin{remark}\label{rem:nondeg}
In the nondegenerate case $\tilde a=a$ and $\tilde b=b>0$, without damping ($\alpha=0$), by Theorems~\ref{thm:mainlow} and~\ref{thm:mainhi} we get
\begin{equation}\label{eq:best}
\| \rho^\mu e^{it\omega(\rho)}\|_{M_p^q} \lesssim \begin{cases}
t^{-\frac{\kappa}{b}}, & t\in(0,1], \\
t^{-\min\left\{\frac{\kappa}{a},\left(\frac{n-1}\tau +\frac1m\right)d\right\}}, & t\geq1,
\end{cases}
\end{equation}
consistently with the result in~\cite[Theorem 3.4]{CMY} for $q=p'$ and $p\in(1,2]$, when $m=2$.
\end{remark}
The presence of a dissipation generates an exponential decay in time at high frequencies, since $e^{-t\alpha(\rho^2)}\leq e^{-ct}$ for $\rho\geq N_1/2$, due to $\alpha\sim\rho^\theta$, with $\theta\geq0$.

Moreover, if $\theta>0$, the dissipation also generates a smoothing effect.
\begin{theorem}\label{thm:mainhiXi}
In the same assumptions of Theorem~\ref{thm:mainhi}, assume that~\eqref{eq:thetalow} holds for some $\theta\in(0,b]$ for any $\rho\geq N_0/2$. Then
\begin{equation}\label{eq:esthi0Xideg}
\| \rho^\mu e^{it\omega(\rho)-t\alpha(\rho^2)}\chi_1(\rho) \|_{M_p^q} \leq \begin{cases}
C\,e^{-ct}, & t\geq1, \\
C\,t^{-\Xi-\frac{\kappa-\Xi\theta}{b}},& t\in(0,1],
\end{cases}
\end{equation}
where
\begin{equation}\label{eq:Xi}
\Xi\theta = \begin{cases}
\left(\kappa-\left(\frac{n-1}\tau + \frac1m\right) db\right)_+, & \tilde b=b,\\
\left( \kappa-\frac{n-1}\tau db\right)_+, & \tilde b<0<b.
\end{cases}
\end{equation}
\end{theorem}
We postpone the proof of Theorem~\ref{thm:mainhiXi} to~\textsection\ref{sec:theoremproofhi}.

We notice that $\Xi=0$ in~\eqref{eq:Xi} if either $\tilde b=b$ and~\eqref{eq:reg} holds, or $\tilde b<0<b$ and~\eqref{eq:shrink} holds.
\begin{remark}
In Theorem~\ref{thm:mainhiXi} we assumed $\theta\in(0,b]$, in view of Remark~\ref{rem:theta}. Indeed, when $\rho\,\omega'(\rho)\sim \rho^b$ for some $b>0$ as $\rho\to\infty$, it generally happens (under reasonable regularity assumptions) that $\omega(\rho)\sim \rho^b$ as well. 
\end{remark}

\section{Application to PDE}\label{sec:models}

In this section, we consider the application of Theorems~\ref{thm:mainlow}, \ref{thm:mainhi}, and~\ref{thm:mainhiXi}, to various models of PDE. For each model, we compute the parameters $a,\tilde a$ appearing in Theorem~\ref{thm:mainlow}, and the parameters $b,\tilde b$ appearing in Theorems~\ref{thm:mainhi} and~\ref{thm:mainhiXi}. We will always fix $m=2$ (see Remark~\ref{rem:m2}).

We anticipate in the following table the values of the parameters $a,\tilde a,b,\tilde b$ in the principal models that we discuss.

\begin{center}
\begin{tabular}{|l|l|c|c|c|c|c|}
  \hline
  Model & equation & \textsection & $a$ & $\tilde a$ & $b$ & $\tilde b$ \\
  \hline
  Plate & \eqref{eq:plate} & \ref{sec:plate} & $2$ & $2$ & $2$ & $2$ \\
  Plate with mass & \eqref{eq:platemass} & \ref{sec:platemass} & $4$ & $4$ & $2$ & $2$ \\
  Boussinesq & \eqref{eq:BQ} & \ref{sec:BQ} & $1$ & $\textbf{3}$ & $2$ & $2$ \\
  Viscoelastic wave & \eqref{eq:VW} & \ref{sec:visco} & $1$ & $\textbf{3}$ & $--$ & $--$ \\
  Klein-Gordon & \eqref{eq:KG} & \ref{sec:KG} & $2$ & $2$ & $1$ & $\textbf{-1}$ \\

  \hline
\end{tabular}
\end{center}
We first briefly consider two plate models that are not degenerate, so they are not the main objective of our studies. Still, our analysis shows the interaction of dispersion and dissipation in two easier cases.

Then we consider Boussinesq and viscoelastic wave models, that are degenerate at low frequencies, and finally the Klein-Gordon equation, that is degenerate at high frequencies.

For given $\ell\in\N$ and $\nu,\nu_0\in\R$, we plan to derive $L^p-L^q$ estimates in the form
\[ \||D|^\nu \partial_t^\ell u (t,\cdot) \|_{L^q} \lesssim \varphi(t)\,\|(\<D\>^{\nu_0} u_0,\<D\>^{\nu_0}\mathscr{F}^{-1}(\omega^{-1}\,\hat u_1))\|_{L^p}, \]
for a suitable $\varphi(t)$. In order to do this, we consider separately low and high frequencies.

We fix $N_1=2N_0$ in~\eqref{eq:chi} so that $\chi_0+\chi_1=1$ and we apply Theorems~\ref{thm:mainlow}, \ref{thm:mainhi} and~\ref{thm:mainhiXi}, with $\mu=\mu(\nu,\ell)$ such that $\xii^\nu\,|i\omega+\alpha|^\ell\approx \rho^\mu$ at low frequencies, and $\mu=\mu(\nu-\nu_0,\ell)$ such that $\xii^{\nu-\nu_0}\,|i\omega-\alpha|^\ell \approx \rho^{\mu}$ at high frequencies.
\begin{assumption}\label{ass:kappa}
Having in mind that $\kappa=n(1/p-1/q)+\mu$, in the following, we write $\kappa_\low$ and $\kappa_\high$ to distinguish low and high frequencies. We recall that Theorem~\ref{thm:mainlow} is applicable if $\kappa_\low>0$, or $\kappa_\low\geq0$ if $1<p\leq 2\leq q<\infty$. When $u_0=0$, we may relax the assumption to $\kappa_\low>-a$, or to $\kappa_\low\geq-a$ if $1<p\leq 2\leq q<\infty$. Similarly, we assume for brevity that $\kappa_\high\geq0$ if $1<p\leq 2\leq q<\infty$, or $\kappa_\high>0$ if $p=1$ or $q=\infty$. We will omit to write the assumptions on $\kappa_\low$ and $\kappa_\high$ in every model.
\end{assumption}
For every model, we will explicitly write the regularity restriction from above on $\kappa_\high$ given by~\eqref{eq:reg} or, respectively, \eqref{eq:shrink}, if $\tilde b=b>0$ or, respectively, $\tilde b<0<b$, that comes into play when there is no dissipation or when the dissipation produces no smoothing ($\theta=0$ in~\eqref{eq:thetalow}).

\subsection{The plate equation}\label{sec:plate}

We consider the initial value problem for the damped plate equation~\eqref{eq:DPl}:
\begin{equation}\label{eq:plate}\tag{P}\begin{cases}
u_{tt}+A^2 u + 2\e A u_t =0, & x\in\R^n, \ t>0,\\
(u,u_t)(0,x)=(u_0,u_1)(x), & x\in\R^n.
\end{cases}\end{equation}
With the notation in~\eqref{eq:abgeneral}, $\alpha=\e\rho^{2}$ and $\beta=\rho^4$, so that $\omega=\sqrt{1-\e^2}\,\rho^2$ is homogeneous; hence, with the notation in Theorems~\ref{thm:mainlow} and~\ref{thm:mainhi}, $a=\tilde a=b=\tilde b=2$; hence, $w''$ is not degenerate. Since $|i\omega-\alpha|^\ell \approx \rho^{2\ell}$, we define
\begin{align}
\label{eq:klowP}
\kappa_\low
    & =n(1/p-1/q)
    +\nu+2\ell,\\
\label{eq:khiP}
\kappa_\high
    & = n(1/p-1/q)
    +\nu-\nu_0+2\ell.
\end{align}
Applying Theorem~\ref{thm:mainlow} without dissipation, the solution to~\eqref{eq:plate} with $\e=0$ verifies the low-frequencies estimate with $\delta_\low$ as in~\eqref{eq:deltalowtheta3}-\eqref{eq:deltaundamped}, that is,
\begin{equation}\label{eq:platelow}
\||D|^\nu \partial_t^\ell \mathscr{F}^{-1}(\chi_0\hat u)(t,\cdot) \|_{L^q} \lesssim (1+t)^{-\min\left\{\frac{\kappa_\low}2,\left(\frac{n-1}\tau+\frac12\right)d\right\}}\,\|(u_0,|D|^{-2}u_1)\|_{L^p},\qquad t\geq0.
\end{equation}
When~$\e>0$, $\theta=2=a$, so Theorem~\ref{thm:mainlow} holds with $\delta_\low=\kappa_\low/a$, that is,
\[ \||D|^\nu \partial_t^\ell \mathscr{F}^{-1}(\chi_0\hat u)(t,\cdot) \|_{L^q} \lesssim (1+t)^{-\frac{\kappa_\low}2}\,\|(u_0,|D|^{-2}u_1)\|_{L^p},\qquad t\geq0. \]
When $\e=0$, applying Theorem~\ref{thm:mainhi} at high frequencies, we get
\begin{equation}\label{eq:Phi}
\||D|^\nu \partial_t^\ell \mathscr{F}^{-1}(\chi_1\hat u)(t,\cdot) \|_{L^q}\lesssim \|(\<D\>^{\nu_0}u_0,\<D\>^{\nu_0-2}u_1)\|_{L^p}\times\begin{cases}
t^{-\frac{\kappa_\high}2}, & t\in(0,1], \\
t^{-\left(\frac{n-1}\tau+\frac12\right)d}, & t\geq1,
\end{cases}
\end{equation}
provided that~\eqref{eq:reg} holds, that is,
\begin{equation}\label{eq:regP}
\frac{\kappa_\high}2 \leq \left(\frac{n-1}\tau + \frac12\right) d.
\end{equation}
In particular, the solution to~\eqref{eq:plate} verifies the long time estimate
\begin{equation}\label{eq:platedecay}
\||D|^\nu \partial_t^\ell u(t,\cdot) \|_{L^q} \lesssim t^{-\min\left\{\frac{\kappa_\low}2,\left(\frac{n-1}\tau+\frac12\right)d\right\}}\,\|(\<D\>^{\nu_0} u_0,\<D\>^{\nu_0}|D|^{-2}u_1)\|_{L^p}, \quad t\geq1.
\end{equation}
Letting $\nu_0=n\left(\frac1p-\frac1q\right)+\nu+2\ell$ (or $\nu_0>\ldots$ if $p=1$ or $q=\infty$), we may remove the singularity at~$t=0$, so that we may replace $t$ by $1+t$ and $t\geq1$ by $t\geq0$ in the right-hand side of~\eqref{eq:platedecay}.

When $\e>0$, the dissipation produces exponential decay and smoothing; applying Theorem~\ref{thm:mainhiXi} with $\theta=b=2$, we find:
\begin{equation}\label{eq:platedecaydamp}
\||D|^\nu \partial_t^\ell u(t,\cdot) \|_{L^q} \lesssim \|(\<D\>^{\nu_0} u_0,\<D\>^{\nu_0}|D|^{-2}u_1)\|_{L^p} \times\begin{cases}
t^{-\frac{\kappa_\high}2}, & t\in(0,1],\\
t^{-\frac{\kappa_\low}2}, & t\geq1,
\end{cases}
\end{equation}
with no need to assume~\eqref{eq:regP}.
\begin{example}
Fix~$\e=0$ and let $A$ be as in~\eqref{eq:Aell} with $\ell=2$ in~\eqref{eq:plate}, that is, we consider
\begin{equation*}\begin{cases}
u_{tt}+ \sum_{|\gamma|=4}c_\gamma \partial_x^\gamma u =0, & x\in\R^n, \ t>0,\\
(u,u_t)(0,x)=(u_0,u_1)(x), & x\in\R^n.
\end{cases}\end{equation*}
We fix $\ell=0$, $\nu=-2$ and $q=p'$, so that $d=\frac2p-1$ and $\kappa_\low=nd-2$. As in Assumption~\ref{ass:kappa}, we set $nd\geq2$ if $p>1$ or $nd>2$ if $p=1$; in the case $u_0=0$, we may omit this assumption. Then, replacing $(u_0,|D|^{-2}u_1)$ by $(|D|^2u_0,u_1)$, \eqref{eq:platedecay} gives
\[ \|u(t,\cdot) \|_{L^{p'}} \lesssim t^{-\min\left\{\frac{nd}2-1,\left(\frac{n-1}\tau+\frac12\right)d\right\}}\,\|(\<D\>^{\nu_0} |D|^2u_0,\<D\>^{\nu_0}u_1)\|_{L^p}, \quad t\geq1. \]
In particular, the decay is $t^{1-n\left(\frac1p-\frac12\right)}$ if $\tau=2$. When $\tau=4$ as in Examples~\ref{ex:CMY} and~\ref{ex:DFZ}, the decay is $t^{-\frac{n+1}2\left(\frac1p-\frac12\right)}$ when $(n-1)(1/p-1/2)\geq2$.
\end{example}

\subsection{The plate equation with mass}\label{sec:platemass}

We consider the initial value problem for the damped plate equation with mass~\eqref{eq:DPlmass}:
\begin{equation}\label{eq:platemass}\tag{PM}\begin{cases}
u_{tt}+A^2 u + u + 2\e A^\gamma u_t =0, & x\in\R^n, \ t>0,\\
(u,u_t)(0,x)=(u_0,u_1)(x), & x\in\R^n,
\end{cases}\end{equation}
where $\gamma=0$ or $\gamma=1$. With the notation in~\eqref{eq:abgeneral}, $\alpha=\e\rho^{2\gamma}$, $\beta=1+\rho^4$, so that
\[ \omega=\sqrt{1+\rho^4-\e^2\rho^{4\gamma}}.\]
By straightforward computation, when $\gamma=0$ we get
\begin{equation}\label{eq:derDPLtheta0}
\omega'(\rho)
    = \frac{2\rho^3}{(1-\e^2+\rho^4)^{\frac12}}>0,\qquad
\omega''(\rho)
    = \frac{6(1-\e^2)\rho^2+2\rho^6}{(1-\e^2+\rho^4)^{\frac32}}>0,
\end{equation}
while, when $\gamma=1$, we obtain
\begin{equation}\label{eq:derDPLtheta1}
\omega'(\rho)
    = (1-\e^2)\,\frac{2\rho^3}{(1+(1-\e^2)\rho^4)^{\frac12}}>0,\qquad
\omega''(\rho)
    = (1-\e^2)\,\frac{6\rho^2+2(1-\e^2)\rho^6}{(1+(1-\e^2)\rho^4)^{\frac32}}>0,
\end{equation}
provided that $\e\in[0,1)$. With the notation in Theorems~\ref{thm:mainlow} and~\ref{thm:mainhi}, %
we get $a=\tilde a=4$ and $b=\tilde b=2$. In particular, $w''$ is not degenerate. Since $|i\omega-\alpha|^\ell \approx \<\rho\>^{2\ell}$, we define
\begin{align}\label{eq:klowPM}
\kappa_\low
    & = n(1/p-1/q)
    +\nu,\\
\label{eq:khiPM}
\kappa_\high
    & = n(1/p-1/q)
    +\nu-\nu_0+2\ell.
\end{align}
We notice that $\kappa_\high$ is as in~\eqref{eq:khiP}, while $\kappa_\low$ is independent on $\ell$, due to the presence of the mass term.

First let $\gamma=0$ and $\e\in[0,1)$, so that the influence of the dissipation amounts to a multiplication by $e^{-\e t}$. Applying Theorem~\ref{thm:mainlow}, the solution to~\eqref{eq:platemass} with $\e=0$ verifies the low-frequencies estimate with $\delta_\low$ as in~\eqref{eq:deltalowtheta3}-\eqref{eq:deltaundamped}, that is,
\[ \||D|^\nu \partial_t^\ell \mathscr{F}^{-1}(\chi_0\hat u)(t,\cdot) \|_{L^q} \lesssim e^{-\e t}\,(1+t)^{-\min\left\{\frac{\kappa_\low}4,\left(\frac{n-1}\tau+\frac12\right)d\right\}}\,\|(u_0,u_1)\|_{L^p},\qquad t\geq0. \]
The high frequencies estimate is as per the plate equation in~\textsection\ref{sec:plate}, multiplied by $e^{-\e t}$, therefore we get
\begin{equation}\label{eq:PMhi}
\||D|^\nu \partial_t^\ell u(t,\cdot) \|_{L^q} \lesssim\|(\<D\>^{\nu_0}u_0,\<D\>^{\nu_0-2}u_1)\|_{L^p}\times\begin{cases}
t^{-\frac{\kappa_\high}2}, & t\in(0,1], \\
e^{-\e t}\,t^{-\min\left\{\frac{\kappa_\low}4,\left(\frac{n-1}\tau+\frac12\right)d\right\}}, & t\geq1,
\end{cases}
\end{equation}
provided that~\eqref{eq:regP} holds.

Now, let $\gamma=1$. Then we may apply Theorem~\ref{thm:mainlow} with $\theta=2<a$ when $\gamma=1$:
\[ \||D|^\nu \partial_t^\ell \mathscr{F}^{-1}(\chi_0\hat u)(t,\cdot) \|_{L^q} \lesssim (1+t)^{-\frac{\kappa_\low}2}\,\|(u_0,u_1)\|_{L^p},\qquad t\geq0. \]
Moreover, we have exponential decay and smoothing at high frequencies. By applying Theorem~\ref{thm:mainhiXi} with $\theta=2=b$, we may remove the assumption~\eqref{eq:regP} and get \eqref{eq:platedecaydamp} (the only difference being that $+2\ell$ does not appear in the definition of $\kappa_\low$ when there is the mass).

In both cases $\gamma=0,1$, the singularity may be removed setting $\nu_0=n\left(\frac1p-\frac1q\right)+\nu$ or $\nu_0>\ldots$ if $p=1$ or $q=\infty$.
%
%

Our result consistently extends the estimates recently obtained in \cite{AAEL} for $A=-\Delta$ (so that $\tau=2$) and $u_0=0$, $\nu=\ell=0$, $\e=0$. Our result when $\e=0$ and $A=-\Delta$ also improves the results obtained in~\cite{Guo2008} and in~\cite{Levandosky1998}.
%

\subsection{Boussinesq Equation}\label{sec:BQ}

We consider the initial value problem for the Viscous Boussinesq equation~\eqref{eq:VBQ}:
\begin{equation}\label{eq:BQ}\tag{BQ}
\begin{cases}
u_{tt} +A u + A^2 u + 2\e Au_t =0 & t>0,\ x\in\R^n,\\
(u,u_t)(0,x)=(u_0,u_1)(x).
\end{cases}
\end{equation}
With the notation in~\eqref{eq:abgeneral}, $\alpha=\e \rho^2$ and $\beta=\rho^2+\rho^4$, so that
\[ \omega = \rho\,\sqrt{1+\rho^2(1-\e^2)},\]
and, by straightforward computation, we get
\begin{equation}\label{eq:derVBQ}
\omega'
    = \frac{1+2\rho^2(1-\e^2)}{\sqrt{1+\rho^2(1-\e^2)}}>0,\quad
\omega''
    = \rho\,\frac{2\rho^2(1-\e^2)^2+3(1-\e^2)}{(\rho^2(1-\e^2)+1)^\frac32}>0,
\end{equation}
provided that $\e\in[0,1)$, for any $\rho>0$. The case $\e=0$ corresponds to the inviscid Boussinesq model. With the notation in Theorems~\ref{thm:mainlow} and~\ref{thm:mainhi}, $a=1$, $\tilde a=3$, $b=\tilde b=2$. In particular, $\omega''$ is degenerate at low frequencies. Let $\ell\in\N$ and $\nu,\nu_0\in\R$. Since $|i\omega+\alpha|^\ell \approx (\rho\<\rho\>)^{\ell}$, we define
\begin{align}\label{eq:klowBQ}
\kappa_\low
    & = n(1/p-1/q)
    +\nu+\ell,\\
\label{eq:khiBQ}
\kappa_\high
    & = n(1/p-1/q)
    +\nu-\nu_0+2\ell.
\end{align}
We stress that $\kappa_\high$ is as in~\eqref{eq:khiP}.
%
%
%
%
When $\e=0$, applying Theorem~\ref{thm:mainlow} with no dissipation, the solution to~\eqref{eq:BQ} verifies the low-frequencies estimate
\begin{equation}\label{eq:BQlow}
\||D|^\nu \partial_t^\ell \mathscr{F}^{-1}(\chi_0\hat u)(t,\cdot) \|_{L^q} \lesssim (1+t)^{-\delta_\BQ}\,\|(u_0,|D|^{-1}u_1)\|_{L^p},\qquad t\geq0,
\end{equation}
where, according to~\eqref{eq:deltalowtheta3}-\eqref{eq:deltaundamped},
\begin{equation}\label{eq:deltaBQ}\begin{split}
\delta_\BQ
    & = \min\left\{\kappa_\low, \frac{\kappa_\low}3+\frac{2(n-1)}{3\tau}d, \left(\frac{n-1}\tau+\frac12\right)d\right\}\\
    & = \begin{cases}
\kappa_\low, & \kappa_\low\leq\frac{n-1}\tau d,\\
\frac{\kappa_\low}3+\frac{2(n-1)}{3\tau}d , & 0\leq \kappa_\low-\frac{n-1}\tau d\leq \frac32 d,\\
\left(\frac{n-1}\tau+\frac12\right)d , & \kappa_\low\geq \left(\frac{n-1}\tau +\frac32\right) d.
\end{cases}
\end{split}
\end{equation}
In the case $\e>0$, the decay rate improves in some range for $\kappa_\low$. Setting $\theta=2\in(a,\tilde a)$, using~\eqref{eq:deltalowtheta2} we find %
%
%
%
%
%
%
%
that~\eqref{eq:BQlow} holds with $\delta_\BQ$ in~\eqref{eq:deltaBQ} replaced by
\begin{equation}\label{eq:deltaVBQ}
\delta_\VBQ = \begin{cases}
\delta_\BQ=\kappa_\low , & \kappa_\low\leq\frac{n-1}\tau d,\\
\frac12\left(\kappa_\low+\frac{n-1}{\tau}d\right) , & \kappa_\low\geq\frac{n-1}\tau d.
\end{cases}
\end{equation}
%
%
Even in the simple case $A=-\Delta$, the decay rate described by~\eqref{eq:deltaVBQ} improves the decay rate obtained in~\cite[Lemma 3.1]{LW} and in \cite[Theorem 2.1]{Dao-Chen=2022}.

We now consider high frequencies. Applying Theorem~\ref{thm:mainhi}, the solution to~\eqref{eq:BQ} with $\e=0$ also verifies the high-frequencies estimate~\eqref{eq:Phi}, provided that~\eqref{eq:regP} holds (plate behavior). %
%
In particular, the solution to~\eqref{eq:BQ} with $\e=0$ verifies the estimate
\begin{equation}\label{eq:decayBQ}
\||D|^\nu \partial_t^\ell u(t,\cdot) \|_{L^q} \lesssim \|(\<D\>^{\nu_0} u_0,|D|^{-1}\<D\>^{\nu_0-1}u_1)\|_{L^p}\times \begin{cases}
t^{-\frac{\kappa_\high}2} , & t\in(0,1],\\
t^{-\delta_\BQ}, & t\geq1.
\end{cases}
\end{equation}
Now let $\e\in(0,1)$. In this case, the dissipative term $e^{-\e t\rho^2(\xi)}$ produces exponential decay and smoothing at high frequencies, as in Theorem~\ref{thm:mainhiXi} with $\theta=2=b$; in particular, without the need to assume~\eqref{eq:regP}, we obtain the estimate
%
\begin{equation}\label{eq:decayVBQ}
\||D|^\nu \partial_t^\ell u(t,\cdot) \|_{L^q} \lesssim \|(\<D\>^{\nu_0}u_0,\<D\>^{\nu_0-2}u_1)\|_{L^p}\times\begin{cases}
C\,t^{-\frac{\kappa_\high}2}  , & t\in(0,1], \\
C\,t^{-\delta_\VBQ}\, , & t\geq1.
\end{cases}
\end{equation}
For both the inviscid and the viscous model, the singularity at $t=0$ is removed setting $\nu_0=n\left(\frac1p-\frac1q\right)+\nu+2\ell$, or $\nu_0>\ldots$ if $p=1$ or $q=\infty$.
\begin{remark}\label{rem:dispersiveBQ}
When $\e=0$, we are particularly interested in the case whether
\[ \delta_\BQ=\left(\frac{n-1}\tau+\frac12\right)d. \]
This case corresponds to the scenario where the decay estimate only depends on $d$ and not on $\kappa_\low$. In particular, since $\kappa_\low\geq nd+\nu+\ell$, then $\delta_\BQ$ is as above when
\begin{equation}\label{eq:puredispersiveBQ}
\nu+\ell + \left(n-\frac{n-1}\tau-\frac32\right)d\geq0.
\end{equation}
The above condition reduces to $n\geq2$ %
%
%
when $\nu+\ell=0$.
\end{remark}
%
%
In the following remarks, we fix $\tau=2$ and $\e=0$ to show the consistency of our estimates with results known in literature in the case $A=-\Delta$.
\begin{remark}\label{rem:BQ1inf}
Let $\nu=\ell=0$. Then, according to Remark~\ref{rem:dispersiveBQ}, $\delta_\BQ=(n/2)d$ for any $n\geq2$. In particular, when $p=1$, $q=\infty$, so that $d=1$, we get
\begin{equation}\label{eq:L1infBQ} \delta_\BQ=\begin{cases}
\frac13 , & n=1,\\
\frac{n}2 , & n\geq2.
\end{cases}
\end{equation}
The $L^1-L^\infty$ estimate (for $\nu_0=0$)
\[ \|u(t,\cdot) \|_{L^\infty} \lesssim t^{-\delta_\BQ} \,\|(u_0,|D|^{-1}\<D\>^{-1}u_1)\|_{L^1},\qquad t\geq1, \]
with $\delta_\BQ$ as in~\eqref{eq:L1infBQ}, has been derived for $n\geq2$ in \cite{Liu2019JDE} (see Lemma 3.1) and for $n=1$ in~\cite{YLiu1993}.

In the case of the 1-dimensional Boussinesq equation, Y. Cho and T. Ozawa~\cite{Cho-Ozawa=2007} used the ``vanishing condition'' to improve the decay rate to $t^{-\frac12}$. This corresponds, in our Remark~\ref{rem:dispersiveBQ}, to fix $\nu=1/2$, so that~\eqref{eq:puredispersiveBQ} holds for any $d$, and to replace $(u_0,|D|^{-1}u_1)$ by $(|D|^{-\frac12}u_0,|D|^{-\frac32}u_1)$; in the case $p=1$, $q=\infty$ and $\ell=0$ this gives
\[ \|u(t,\cdot) \|_{L^\infty} \lesssim t^{-\frac12} \,\|(|D|^{-\frac12}\,\<D\>^{\frac12}\,u_0,|D|^{-\frac32}\<D\>^{-\frac12}u_1)\|_{L^1},\qquad t\geq1. \]
%
\end{remark}
\begin{remark}\label{rem:BQwave}
Let $u_0=0$. Replacing $u_1$ by $|D|u_1$, \eqref{eq:BQlow} becomes
\begin{equation}\label{eq:BQlowu1}
\||D|^{\nu+1} \partial_t^\ell \mathscr{F}^{-1}(\chi_0\hat u)(t,\cdot) \|_{L^q} \lesssim (1+t)^{-\delta_\BQ}\,\|u_1\|_{L^p}, \qquad t\geq0.
\end{equation}
%
In particular, let $\nu=-1$ and $\ell=0$, so that $\kappa_\low=n(1/p-1/q)-1$ (Assumption~\ref{ass:kappa} holds due to $u_0=0$) and, according to~\eqref{eq:deltalowtheta3}-\eqref{eq:deltaundamped},
%
%
%
\[ \delta_\BQ=n\left(\frac1p-\frac1q\right)-1 \iff n\left(\frac1p-\frac1q\right)- (n-1)\frac{d}2 \leq 1. \]
The $(p,q)$ region determined by the above inequality is the same in which the solution operator for the classic wave equation ($A=-\Delta$ in~\eqref{eq:W}) is bounded from $L^p$ to $L^q$ (see~\cite{P,Strichartz1970}). In that region, the Boussinesq equation satisfies, for $t\geq1$, the same estimate of the classic wave equation, i.e.,
\[ \|u(t,\cdot) \|_{L^q} \lesssim t^{1-n\left(\frac1p-\frac1q\right)} \,\|u_1\|_{L^p},\qquad t>0. \]
At short times $t\in(0,1]$, however, the estimate for the Boussinesq equation ($\nu=-1$ and $\nu_0=1$ in~\eqref{eq:decayBQ}) gives $t^{1-\frac{n}2\left(\frac1p-\frac1q\right)}$.
\end{remark}
\begin{remark}
In~\cite{Cho-Ozawa=2007} the authors explained that the same estimate for the Boussinesq equation in space dimension $n=1$ can be obtained by using the third-order stationary phase method. This corresponds to set $m=3$ in Theorem~\ref{thm:mainlow}. However, applying Theorem~\ref{thm:mainlow} with $m=3$ in place of $m=2$ does not bring any benefit for the Boussinesq equation (and, more in general, for any model with a phase function $w$ that is also smooth at $\rho=0$), see Remark~\ref{rem:m2}. Roughly speaking, we can use a second-order stationary phase method, and we do not need higher order methods, since Lemma~\ref{lem:low} already takes into account of an interpolation with the case that formally corresponds to infinite order method. 
\end{remark}

\subsubsection{A variant of the viscous Boussinesq equation}\label{sec:BQvar}

We consider the initial value problem for the following variant of the Viscous Boussinesq equation~\eqref{eq:VBQ}, considered in~\cite{CT25,C02,CMP07,CV94}:
\begin{equation}\label{eq:VVBQ}\tag{VVBQ}
\begin{cases}
u_{tt} +A u + A^2 u + 2\e A^2 u_t =0 , & t>0,\ x\in\R^n,\\
(u,u_t)(0,x)=(u_0,u_1)(x),
\end{cases}
\end{equation}
for $\e>0$. With the notation in~\eqref{eq:abgeneral}, $\alpha=\e \rho^4$ and $\beta=\rho^2+\rho^4$, so that
\[ \omega = \rho\,\sqrt{1+\rho^2-\e^2\rho^6},\]
and, %
%
by straightforward computation,
\begin{equation}\label{eq:derVVBQ}
\omega'
    = \frac{1+2\rho^2-4\e^2\rho^6}{\sqrt{1+\rho^2-\e^2\rho^6}}>0,\qquad
\omega''
    = \rho\,\frac{3+2\rho^2-3\e^2\rho^4(7+6\rho^2)+12\e^4\rho^{10}}{(1+\rho^2-\e^2\rho^6)^\frac32}>0,
\end{equation}
provided that $\rho$ is sufficiently small with respect to $\e^{-1}$. With the notation in Theorem~\ref{thm:mainlow}, $a=1$ and $\tilde a=3$, as per the Boussinesq equation, and $\theta=4>\tilde a$; hence, we define $\kappa_\low$ as in~\eqref{eq:klowBQ} and we obtain %
%
%
%
%
%
%
%
that~\eqref{eq:BQlow} holds with $\delta_\BQ$ in~\eqref{eq:deltaBQ} replaced by $\delta_\low$ as in~\eqref{eq:deltalowtheta3} with $\theta=4$, that is,
\begin{equation}\label{eq:deltaVVBQ}
\delta_\VVBQ = \begin{cases}
\delta_\BQ=\kappa_\low , & \kappa_\low\leq\frac{n-1}\tau d,\\
\frac{\kappa}3+\frac23\,\frac{n-1}\tau d , & 0\leq\kappa_\low-\frac{n-1}\tau d\leq\frac32\, d,\\
\frac\kappa4 + \frac34\frac{n-1}\tau d + \frac18\,d , & \kappa_\low \geq \left(\frac{n-1}\tau +\frac32\right) d.
\end{cases}
\end{equation}


\subsection{The viscoelastic wave equation}\label{sec:visco}

We consider the initial value problem for the wave equation with viscoelastic damping~\eqref{eq:Wvisco}:
\begin{equation}\label{eq:VW} \begin{cases}\tag{VW}
u_{tt} + A u + 2\e Au_t =0, & t>0,\ x \in\R^n,\\
(u,u_t)(0,x)=(u_0,u_1)(x).
\end{cases} \end{equation}
With the notation in~\eqref{eq:abgeneral}, $\alpha=\e\rho^2$ and $\beta=\rho^2$, so that, by straightforward computation,
%
\begin{equation}\label{eq:derVW}
\omega
    =\sqrt{\rho^2-\e^2\rho^4},\quad
\omega'
    = \frac{\rho-2\e^2\rho^3}\omega>0,\quad
\omega''
    = \frac{-3\e^2\rho^4+2\e^4\rho^6}{\omega^3}<0,
\end{equation}
provided that $2\e^2\rho^2<1$. With the notation in Theorem~\ref{thm:mainlow}, $a=1$, $\tilde a=3$, $\theta=2\in(a,\tilde a)$. In particular, $\omega''$ has the same degeneracy at low frequencies of the Boussinesq equation~\eqref{eq:BQ}. Let $\ell\in\N$ and $\nu,\nu_0\in\R$. Since $|i\omega+\alpha|^\ell \approx \rho^{\ell}$ as $\rho\to0$, we define $\kappa_\low=n(1/p-1/q)+\nu+\ell$ as in~\eqref{eq:klowBQ}. Applying Theorem~\ref{thm:mainlow}, the solution to~\eqref{eq:VW} verifies the low-frequencies estimate~\eqref{eq:BQlow} with $\delta_\BQ$ replaced by~$\delta_\VBQ$ as in~\eqref{eq:deltaVBQ}.
\begin{remark}
Let $q=p'$ and $\ell=0$, so that $\kappa_\low=nd+\nu$. When $\kappa_\low \leq \frac{n-1}\tau d$, the estimate we obtain for the viscoelastic wave, namely,
\[ \||D|^\nu\mathscr{F}^{-1}(\chi_0\hat u)(t,\cdot) \|_{L^q} \lesssim (1+t)^{-\min\left\{\kappa_\low,\frac12\left(\kappa_\low+\frac{n-1}{\tau}d\right)\right\}}\,\|(u_0,|D|^{-1}u_1)\|_{L^p},\quad t\geq0,\]
is the same estimate obtained for the free wave (see~\cite[Theorem 3.1]{CMY}).
\end{remark}


\subsubsection{The structurally damped wave equation}\label{sec:structural}

As a variant of the wave equation with viscoelastic dissipation, we consider the wave equation with structural damping~\eqref{eq:Wvisco}:
\begin{equation}\label{eq:SW} \begin{cases}\tag{SW}
u_{tt} + A u + 2\e A^{\frac12} u_t =0, & t>0,\ x \in\R^n,\\
(u,u_t)(0,x)=(u_0,u_1)(x).
\end{cases} \end{equation}
With the notation in~\eqref{eq:abgeneral}, $\alpha=\e\rho$ and $\beta=\rho^2$, so that $\omega=\sqrt{1-\e^2}\,\rho$, provided that $\e\in(0,1)$. The degeneracy for this model is the same of the wave equation, that is, with the notation in Theorem~\ref{thm:mainlow}, $a=b=1=\theta$ and, formally, $\tilde a=m=\tilde b=\infty$. Therefore, $\delta_\low=\kappa_\low$. At high frequencies, the dissipation produces exponential decay and smoothing. Therefore, for $\nu_0=0$ we get
\[ \||D|^\nu \partial_t^\ell u(t,\cdot) \|_{L^q} \leq C\,t^{-n\left(\frac1p-\frac1q\right)-\nu-\ell}\,\|(u_0,|D|^{-1}u_1)\|_{L^p}, \qquad t>0,\]
consistently with the result obtained in~\cite{DAE14JDE,DAE17,DAR14,PKR15}.

\subsection{The Klein-Gordon equation}\label{sec:KG}

We consider the initial value problem for the damped Klein-Gordon equation~\eqref{eq:DKG}:
\begin{equation}\label{eq:KG}\tag{KG}
\begin{cases}
u_{tt} +A u + u + 2\e u_t =0 , & t>0,\ x\in\R^n,\\
(u,u_t)(0,x)=(u_0,u_1)(x).
\end{cases}
\end{equation}
With the notation in~\eqref{eq:abgeneral}, $\alpha=\e$ and $\beta=\rho^2+1$, so that, by straightforward computation,
\begin{equation}\label{eq:derDKG}
\omega
    = \sqrt{\rho^2+1-\e^2},\quad
\omega'
    = \frac{\rho}{\sqrt{\rho^2+1-\e^2}}>0,\quad
\omega''
    = (1-\e^2)\,\frac{\sqrt{\rho^2+1-\e^2}}{(\<\rho\>^2-\e^2)^2}>0,
\end{equation}
provided that $\e^2<1$, for any $\rho>0$. The case $\e=0$ corresponds to the Klein-Gordon model, while the case $\e>0$ is sometimes called telegraph equation. With the notation in Theorems~\ref{thm:mainlow} and~\ref{thm:mainhi}, $a=\tilde a=2$, $b=1$, and $\tilde b=-1$. In particular, $\omega''$ is degenerate at high frequencies. Let $\ell\in\N$ and $\nu,\nu_0\in\R$. Since $|i\omega+\alpha|^\ell \approx \<\rho\>^\ell$, we define
\begin{align}\label{eq:klowKG}
\kappa_\low
    & = n(1/p-1/q)
    +\nu,\\
\label{eq:khiKG}
\kappa_\high
    & = n(1/p-1/q)
    +\nu-\nu_0+\ell.
\end{align}
We notice that $\kappa_\low$ is as in~\eqref{eq:klowPM}. The influence from the dissipation means that the estimates for the solution are multiplied by $e^{-\e t}$. Applying Theorem~\ref{thm:mainlow}, the solution to~\eqref{eq:KG} verifies the low-frequencies estimate with $\delta_\low$ given by~\eqref{eq:deltalowtheta3}-\eqref{eq:deltaundamped}, that is,
\begin{equation}\label{eq:KGlow}
\||D|^\nu \partial_t^\ell \mathscr{F}^{-1}(\chi_0\hat u)(t,\cdot) \|_{L^q} \lesssim e^{-\e t}\,(1+t)^{-\min \left\{\frac{\kappa_\low}2,\left(\frac{n-1}\tau +\frac12\right) d\right\}}\,\|(u_0,u_1)\|_{L^p},\qquad t\geq0.
\end{equation}
We consider high frequencies. Condition~\eqref{eq:shrink} reads as
\begin{equation}\label{eq:shrinkKG}
\kappa_\high\leq \frac{n-1}\tau d.
\end{equation}
Applying Theorem~\ref{thm:mainhi}, the solution to~\eqref{eq:KG} verifies the high frequencies estimate
\begin{equation}\label{eq:KGhi}\begin{split}
& \||D|^\nu \partial_t^\ell \mathscr{F}^{-1}(\chi_1\hat u)(t,\cdot) \|_{L^q} \\
& \lesssim \|(\<D\>^{\nu_0} u_0,\<D\>^{\nu_0-1}u_1)\|_{L^p}\times\begin{cases}
t^{-\kappa_\high} , & t\in(0,1], \\
e^{-\e t}\,t^{-\min\left\{\left(\frac{n-1}\tau + \frac12\right)d, 2\frac{n-1}\tau d-\kappa_\high\right\}}, & t\geq1.
\end{cases}
\end{split}
\end{equation}
The singularity at~$t=0$ may be removed setting $\nu_0=n\left(\frac1p-\frac1q\right)+\nu+\ell$ or $\nu_0>\ldots$ if $p=1$ or $q=\infty$.

In particular, the solution to~\eqref{eq:KG} verifies the long time estimate
\begin{equation}\label{eq:KGest}\begin{split}
& \||D|^\nu \partial_t^\ell u(t,\cdot) \|_{L^q} \lesssim e^{-\e t}\,t^{-\delta_\KG}\, \|(\<D\>^{\nu_0} u_0,\<D\>^{\nu_0-1}u_1)\|_{L^p}, \quad t\geq1,\\
& \delta_\KG = \min\left\{\kappa_\low,\left(\frac{n-1}\tau + \frac12\right)d, 2\frac{n-1}\tau d-\kappa_\high \right\}.
\end{split}
\end{equation}
\begin{remark}
Let $q=p'$ and $\ell=0$, so that $\kappa_\low=nd+\nu$ and $\kappa_\high=nd+\nu-\nu_0$. Assume that $\frac{n-1}\tau\geq\frac12$, and that (3.21) in~\cite{CMY} holds, namely,
\[ nd-\frac{n-1}\tau d +\frac12 d \leq \nu_0-\nu, \]
that is, $\kappa_\high\leq \frac{n-1}\tau d -\frac12 d$. Then, for $\nu\geq0$,
\[ \delta_\KG = \left(\frac{n-1}\tau + \frac12\right)d, \]
as in~\cite[Theorem 3.2]{CMY} for $\epsilon=0$. As per the result obtained in~\cite{CMY}, our estimates are consistent with the estimates obtained in \cite{MarshallStrauss1980} for $A=-\Delta$ and $u_0=0$ (we include the endpoint cases $p=1$ and $q=\infty$, that are excluded in~\cite{CMY}).
\end{remark}

\subsubsection{The Klein-Gordon equation with structural dissipation}

As a variant of the dissipative Klein-Gordon equation, we consider the Klein-Gordon equation with structural dissipation
\begin{equation}\label{eq:SKG}\tag{SKG}
\begin{cases}
u_{tt} +A u + u + 2\e A^{\frac12} u_t =0 , & t>0,\ x\in\R^n,\\
(u,u_t)(0,x)=(u_0,u_1)(x).
\end{cases}
\end{equation}
With the notation in~\eqref{eq:abgeneral}, $\alpha=\e\rho$ and $\beta=\rho^2+1$, so that, by straightforward computation,
\begin{equation}\label{eq:derSKG}
\begin{split}
\omega
    & = \sqrt{\rho^2(1-\e^2)+1},\quad \omega'
    = \frac{\rho(1-\e^2)}{\sqrt{\rho^2(1-\e^2)+1}}>0,\\
\omega''
    & = (1-\e^2)\,\frac{\sqrt{\rho^2(1-\e^2)+1}}{\<\rho\>^4-2\e^2(\rho^2+\rho^4)+\e^4\rho^4}>0,
\end{split}
\end{equation}
provided that $\e^2<1$, for any $\rho>0$. With the notation in Theorems~\ref{thm:mainlow} and~\ref{thm:mainhi}, $a=\tilde a=2$, $b=1$, $\tilde b=-1$, $\theta=1$. At low frequencies, we apply Theorem~\ref{thm:mainlow} with $\theta=1<a$. 
%
%
At high frequencies, we get exponential decay and smoothing; since $\theta=1=b$, we obtain
\begin{equation}\label{eq:SKGest}
\||D|^\nu \partial_t^\ell u(t,\cdot) \|_{L^q} \lesssim \|\<D\>^{\nu_0}(u_0,u_1)\|_{L^p}\times \begin{cases}
t^{-\kappa_\high} , & t\in(0,1],\\
t^{-\kappa_\low} , & t\geq1,
\end{cases}
\end{equation}
with $\kappa_\low$ as in~\eqref{eq:klowKG} and $\kappa_\high$ as in~\eqref{eq:khiKG}, with no need to assume~\eqref{eq:shrinkKG}.

\section{Application to the nonlinear Boussinesq equation}\label{sec:nonlinear}

In this section, we consider the Cauchy problem for semilinear Boussinesq equation
\begin{equation}\label{eq:NBQ}\tag{NBQ}
\begin{cases}
u_{tt}+A u + A^2 u+2\epsilon A  u_t =A f(u) , & t>0,\ x\in\R^n,\\
(u,u_t)(0,x)=(u_0,u_1)(x),
\end{cases}
\end{equation}
in the viscous case $\e>0$ and in the inviscid case $\e=0$. Here
\[ f(0)=0, \quad |f(u)-f(v)|\leq C\,|u-v|\big(|u|^{\alpha-1}+|v|^{\alpha-1}\big),\]
for some $\alpha>1$.
 A notable example of $f$ verifying the above properties is $f(u)=|u|^\alpha$.

The inviscid Boussinesq equation attracted much attention by many researchers due to the wide applications in the real world (see, for instance, \cite{Bona-Sachs=1988, Cho-Ozawa=2007, Linares1993, {YLiu1993},  YLiu1997}). The classic form
\[ u_{tt}-u_{xx}+u_{xxxx} = (u^2)_{xx}, \]
was first derived by Boussinesq in 1872 in \cite{Boussinesq1872} to describe the propagation of long waves with small amplitude on the surface of shallow water, where $u=u(x,t)$ denotes the elevation of the free fluid surface.

We define the operator
\[ (Fu)(t,\cdot) =\int_0^t A\, E(t-s)\ast f(u(s,\cdot))\,ds,\]
where $E(t)$ is the fundamental solution to the linear problem~\eqref{eq:BQ}, i.e.,
\[ u(t,\cdot)= E(t,\cdot)\ast u_1 + (E_t+2\e E)(t,\cdot)\ast u_0. \]
Since $A$ is a second order homogeneous operator, if we fix $\nu=1$ and $\ell=0$ and we apply~\eqref{eq:BQlow} to $|D|^{-1}u$, we get
\[ \|\mathscr{F}^{-1}(\chi_0\hat u)(t, \cdot) \|_{L^q} \leq C\,(1+t)^{-\delta_\BQ}\,\|(|D|^{-1}u_0,|D|^{-2}u_1)\|_{L^p}, \]
with $\delta_\BQ$ replaced by $\delta_\VBQ$ if $\e>0$. We plan to apply the estimate above the term coming from the nonlinear perturbation. Namely, we fix $p$ and $q$ such that $p\alpha=q$, so that
\[ \||D|^{-2}Af(u)\|_{L^p} \approx \|f(u)\|_{L^p} \approx \|u\|_{L^q}^\alpha. \]
On the other hand, we shall fix an initial data space $\mathcal A$ such that the solution of the corresponding linear problem verify the estimate
\[ \|\mathscr{F}^{-1}(\chi_0\hat u)(t, \cdot) \|_{L^q}\leq C\,(1+t)^{-\delta_\BQ}\, \|(u_0,u_1)\|_{\mathcal A}. \]
Let $p=q'$, so that $n(1/p-1/q)=nd$; to fix the initial data space so that the decay rates are the desired ones (see~\eqref{eq:deltaBQ} and~\eqref{eq:deltaVBQ}), we shall impose
\[ \frac\nu{d} \geq \frac{n-1}\tau-n + \begin{cases}
\frac32 & \text{when $\e=0$,}\\
0 & \text{when $\e>0$.}
\end{cases}\]
In particular, the condition above is satisfied for $\nu=0$ for any $n\geq2$ if $\e=0$ and for any $n\geq1$ if $\e>0$; if $n=1$ and $\e=0$, the above condition holds for $\nu\geq d/2$, consistently with the ``vanishing condition'' introduced in space dimension $n=1$ in~\cite{Cho-Ozawa=2007}.

We first consider the case $\e=0$. In view of Remark~\ref{rem:dispersiveBQ}, due to $\nu=1$, condition~\eqref{eq:puredispersiveBQ} holds for any $n\geq1$; hence, $\delta_\BQ=\left(\frac{n-1}\tau+\frac12\right)d$. Therefore, to maximize the decay rate, we shall fix $p$ and $q$ on the conjugate line, that is,
\[ p=1+\frac1\alpha, \quad q=1+\alpha; \qquad \text{hence,}\quad d(p,q)=\frac{\alpha-1}{\alpha+1}\,. \]
The critical exponent $\alpha_c$ for the existence of global-in-time solutions is obtained when $\alpha\delta_\BQ=1$, that is, when $\alpha_c$ solves
\begin{equation}\label{eq:tauStrauss}
\left(\frac{n-1}\tau+\frac12\right) \frac{\alpha-1}{\alpha+1}=\frac1\alpha.
\end{equation}
The critical exponent identified by~\eqref{eq:tauStrauss} is the Strauss exponent in space dimension $1+(n-1)(2/\tau)$, namely, $\gamma(1+2\tau^{-1}(n-1))$, see~\cite{Strauss}. In the case $\tau=2$, we find the classical Strauss exponent $\gamma(n)$ and our result corresponds to the result obtained for $A=-\Delta$ in~\cite[Theorem 2]{Cho-Ozawa=2007}.
\begin{theorem}\label{thm:NBQ}
Let $n\geq 2$ and $\alpha>\alpha_c$, where $\alpha_c$ is the positive solution to~\eqref{eq:tauStrauss}. Moreover, assume that $\alpha<1+4/(n-2)$ if $n\geq3$. Then for initial data
\begin{equation}\label{eq:dataNBQ}
(u_0,u_1)\in\mathcal{A}:=H^{2, 1+\frac1\alpha}\times \big(\dot H^{-1, 1+\frac1\alpha}\cap L^{1+\frac1\alpha}\big),
\end{equation}
with sufficiently small norm
\[ \|(u_0,u_1)\|_{\mathcal{A}}=\|u_0\|_{H^{2, 1+\frac1\alpha}}+\||D|^{-1}u_1\|_{L^{1+\frac1\alpha}}+\|u_1\|_{L^{1+\frac1\alpha}}, \]
there is a unique global (in time) solution $u\in L^\infty([0, \infty), L^{\alpha+1}(\mathbb{R}^n))$ to \eqref{eq:NBQ} with $\e=0$. Moreover, the solution satisfies the following decay estimates:
\begin{equation}\label{eq:decayNIBQ}
\|\, u(t,\cdot)\|_{L^{\alpha+1}} \lesssim (1+t)^{-\delta_\BQ}\|(u_0,u_1)\|_{\mathcal{A}},\quad \delta_\BQ=\left(\frac{n-1}\tau+\frac12\right) \frac{\alpha-1}{\alpha+1}.
\end{equation}
In space dimension $n=1$ the result holds replacing $\mathcal A$ in~\eqref{eq:dataNBQ} by
\[ \mathcal{A}:=\big(\dot H^{-\frac12, 1+\frac1\alpha}\cap H^{1, 1+\frac1\alpha}\big)\times \big(\dot H^{-\frac32, 1+\frac1\alpha}\cap H^{-1, 1+\frac1\alpha}\big). \]
\end{theorem}
%
%
%
%
%
\begin{proof}
The proof is classical (see~\cite{Strauss} for a similar argument). We define the solution space
\[
X
    = \{u\in L^\infty([0,\infty), L^{\alpha+1}): \ \|u\|_X\leq R\},\quad \|u\|_X
    = \sup_{t>0} (1+t)^{\delta_\BQ}\|u(t)\|_{L^{\alpha+1}},
\]
for $R>0$ that will be fixed later. Let $u^\lin$ be the solution to the linear problem~\eqref{eq:BQ} (with $\e=0$). By applying~\eqref{eq:decayBQ} with $\nu=0$ and $\nu_0=nd$, in view of Remark~\ref{rem:dispersiveBQ}, we get
\[
\|u^\lin(t,\cdot)\|_{L^q} \leq C (1+t)^{-\delta_\BQ} \,\|(\<D\>^{nd}u_0,\<D\>^{nd-1}|D|^{-1}u_1)\|_{L^p},\quad t\geq0.
\]
We notice that $nd<2$ is equivalent to $\alpha<1 +4/(n-2)$ when $n\geq3$, so that from the estimate above we derive
\begin{equation}\label{eq:decayBQ-1}
\|u^\lin(t,\cdot)\|_{L^q} \leq C_1 (1+t)^{-\delta_\BQ} \,\|(u_0,u_1)\|_{\mathcal A}.
\end{equation}
When $n=1$, we apply~\eqref{eq:decayBQ} with $\nu=1/2$ and $\nu_0=3/2$ to $|D|^{-\frac12}u$ to get
\[
\|u^\lin(t,\cdot)\|_{L^q} \leq C (1+t)^{-\delta_\BQ} \,\|(\<D\>^{\frac32}|D|^{-\frac12}u_0,\<D\>^{\frac12}|D|^{-\frac32}u_1)\|_{L^p},\quad t\geq0.
\]
In particular, we fix $R=2C_1\|(u_0,u_1)\|_{\mathcal A}$, so that $\|u^\lin\|_X \leq R/2$. We now prove that the operator
\[ (Nu)(t,\cdot)= u^\lin(t,\cdot) + (Fu)(t,\cdot),\]
is a contraction on~$X$. Let $g(\tau,\cdot)=f(u(\tau,\cdot))-f(v(\tau,\cdot))$. We modify the estimate~\eqref{eq:decayBQ-1} for $t\in(0,1]$, to reduce the regularity of $Ag$ (we set $\nu=\nu_0=0$ in~\eqref{eq:Phi}):
\[ \| E(t-\tau)\ast A g(\tau,\cdot)\|_{L^{\alpha+1}} \leq \begin{cases}
C\,(t-\tau)^{-\delta_\BQ}\,\||D|^{-2}Ag(\tau,\cdot)\|_{L^{1+\frac1\alpha}} , & t-\tau\geq1,\\
C\,(t-\tau)^{-\frac{nd}2}\,\|\<D\>^{-2}Ag(\tau,\cdot)\|_{L^{1+\frac1\alpha}} , & t-\tau\in(0,1].
\end{cases} \]
Recalling that $nd<2$, the singularity $(t-\tau)^{-\frac{nd}2}$ is integrable; hence,
\[ \begin{split}
& \int_0^{t-1} \| E(t-\tau)\ast A g(\tau,\cdot)\|_{L^{\alpha+1}} d\tau \leq C\,\int_0^{t-1} (t-\tau)^{-\delta_\BQ}\,\|g(\tau,\cdot)\|_{L^{1+\frac1\alpha}}\,d\tau,\\
& \int_{t-1}^t \| E(t-\tau)\ast A g(\tau,\cdot)\|_{L^{\alpha+1}} d\tau \leq C\,\int_{t-1}^t (t-\tau)^{-\frac{nd}2}\,\|g(\tau,\cdot)\|_{L^{1+\frac1\alpha}}\,d\tau.
\end{split} \]
Since $u,v\in X$, we get
\[ \begin{split}
\|\<D\>^{-2}Ag(\tau,\cdot)\|_{L^{1+\frac1\alpha}}
    & \leq \||D|^{-2}Ag(\tau,\cdot)\|_{L^{1+\frac1\alpha}} \approx \|g(\tau,\cdot)\|_{L^{1+\frac1\alpha}} \\
    & \leq C\,R^{\alpha-1} \|u-v\|_X\,(1+\tau)^{-\alpha\delta_\BQ}.
\end{split} \]
We notice that $\delta_\BQ \leq (nd)/2<1$; hence, using $\alpha\delta_\BQ>1$, we immediately obtain (the estimate of this kind of integrals goes back at least to~\cite{Segal})
\[ \begin{split}
& \int_0^{t-1} (t-\tau)^{-\delta_\BQ}\,(1+\tau)^{-\alpha\delta_\BQ}\,d\tau \lesssim (1+t)^{-\delta_\BQ},\\
& \int_{t-1}^t (t-\tau)^{-\frac{nd}2}\,(1+\tau)^{-\alpha\delta_\BQ}\,d\tau \approx (1+t)^{-\alpha\delta_\BQ} \leq (1+t)^{-\delta_\BQ}.
\end{split} \]
In turn, we get
\[ \|(Nu-Nv)(t,\cdot)\|_{L^{\alpha+1}}\leq \int_0^t \| E(t-\tau)\ast A g(\tau,\cdot)\|_{L^{\alpha+1}} d\tau \leq C_2\,R^{\alpha-1} \|u-v\|_X\,(1+t)^{-\delta_\BQ}. \]
Recalling that $R=2C_1\|(u_0,u_1)\|_{\mathcal A}$ and letting $\|(u_0,u_1)\|_{\mathcal A}$ sufficiently small to obtain $C_2\,R^{\alpha-1}\leq 1/2$, we derive
\[ \|(Nu-Nv)\|_X \leq \frac12\,\|u-v\|_X. \]
Therefore, $F:X\to X$ and it is a contraction. By Banach's contraction principle, there is a unique $u\in X$ such that $Fu=u$, that is, $u$ is a solution to~\eqref{eq:NBQ} with $\e=0$. Moreover, $\|u\|_X \leq R=2C_1\|(u_0,u_1)\|_{\mathcal A}$, that is, we proved~\eqref{eq:decayNIBQ}.
\end{proof}
%
%
In the dissipative case $\e>0$, it is not obvious whether the best decay rate is obtained along the conjugate line or not. Let us fix $q\leq p'$. Since $\nu=1$ and $\ell=0$, in view of $\kappa_\low\geq 1+nd$, we find from~\eqref{eq:deltaVBQ} that
\[ \delta_\VBQ=\frac12\left(n\left(\frac1p-\frac1q\right)+1 +\frac{n-1}\tau\,d\right). \]
Replacing $q=\alpha p$, we find
\[ \delta_\VBQ(p)=\frac12\left(n\,\frac{\alpha-1}{\alpha p} +1 +\frac{n-1}\tau\,\left(1-\frac2{\alpha p}\right)\right); \]
due to
\[ \delta_\VBQ'(p)=-\frac1{2p^2\alpha}\left(n\,(\alpha-1) - 2\frac{n-1}\tau\right), \]
the maximum is obtained for $p=1+1/\alpha$ if $\alpha-1\leq \frac2\tau\,\frac{n-1}n$. The critical exponent is obtained setting $\alpha\,\delta_\VBQ=1$, that is,
\[ n\,\frac{\alpha-1}{p} +\alpha-2 +\frac{n-1}\tau\,\left(\alpha-\frac2{p}\right) =0. \]
For $p=1+1/\alpha$, $\alpha_c$ is the solution to
\[ \frac1{\alpha_c}=\left(n+\frac{n-1}{\tau}\right)\left(\frac{\alpha_c}{\alpha_c+1}-\frac12\right)+\frac12. \]
Results completely analogous to Theorem~\ref{thm:NBQ} follow, but they are expected to be optimal when $\alpha_c<1+\frac2\tau\,\frac{n-1}n$, that is, $n$ is sufficiently large with respect to $\tau$ to get:
\[ \frac{4}{\tau^2} (n-1)^2 + \frac2\tau (n-1)(n+2) -2n>0. \]
Concerning the initial data assumption, when $n\geq2$, we may take $\mathcal A$ as in~\eqref{eq:dataNBQ}, while for $n=1$, we may reduce the regularity of the data, replacing $\mathcal A$ in~\eqref{eq:dataNBQ} by
\[ \mathcal{A}:= H^{1, 1+\frac1\alpha}\times \dot H^{-1, 1+\frac1\alpha}. \]
\begin{remark}
The method to prove Theorem~\ref{thm:NBQ} may be easily adapted to other models. For instance, we may add the nonlinear perturbation $Af(u)$ to the plate equation~\eqref{eq:plate}, that is, we consider
\begin{equation}\label{eq:Nplate}\begin{cases}
u_{tt}+A^2 u + 2\e A u_t =Af(u), & x\in\R^n, \ t>0,\\
(u,u_t)(0,x)=(u_0,u_1)(x), & x\in\R^n.
\end{cases}\end{equation}
When $\e=0$ we obtain the low frequencies estimate
\begin{equation}\label{eq:Pcrucial}
\|\mathscr{F}^{-1}(\chi_0\hat u)(t, \cdot) \|_{L^q} \leq C\,t^{-\left(\frac{n-1}{\tau}+\frac12\right)d} \|(u_0,|D|^{-2}u_1)\|_{L^p},\quad t\geq1,
\end{equation}
setting $\nu=\ell=0$ in~\eqref{eq:platelow}. Setting $p=1+1/\alpha$ and $q=1+\alpha$, we obtain a result analogous to Theorem~\ref{thm:NBQ} with the same critical exponent $\alpha_c$ defined by the solution to~\eqref{eq:tauStrauss}. When $\e>0$, the decay rate is improved:
\begin{equation}\label{eq:Pcrucialeps}
\|\mathscr{F}^{-1}(\chi_0\hat u)(t, \cdot) \|_{L^q} \leq C\,t^{-\frac{n}2\left(\frac1p-\frac1q\right)} \|(u_0,|D|^{-2}u_1)\|_{L^p},\quad t\geq1.
\end{equation}
Therefore, it is more convenient to fix $p=1$ and $q=\alpha$, obtaining global small data solutions when $\alpha>1+2/n$, the Fujita critical exponent, in space dimension $n=1,2$ (since we need $q\geq2$), the same of the nonlinear heat equation. This latter phenomenon is expected, since the solution to the equation in~\eqref{eq:Nplate} behaves as the solution to $A^2 u + 2\e A u_t =Af(u)$ (the so-called ``diffusion phenomenon''), that is, as the solution to the heat equation $A u + 2\e u_t =f(u)$.
\end{remark}


\section{Proof of low frequencies estimates}\label{sec:theoremprooflow}

We first prove Theorem~\ref{thm:mainlow} when $\alpha=0$. More precisely, we prove the following.
\begin{theorem}\label{thm:mainlowdisp}
For any~$\xi$ with $\rho(\xi)\leq N_0$, we assume the following:
\begin{itemize}
\item that~\eqref{eq:wprimelow} and~\eqref{eq:wboundlow} hold for some $a>0$;
\item that~\eqref{eq:derlow} holds for some $m\geq2$ and $\tilde a\geq a$.
\end{itemize}
Let $1\leq p\leq 2\leq q\leq\infty$, $d=d(p,q)$ as in~\eqref{eq:d} and $\kappa$ as in~\eqref{eq:kappa}. We assume that $\kappa>0$. Then,
\begin{equation}\label{eq:estlowdisp}
\| \rho^\mu\,e^{it\omega(\rho)}\chi_0(\rho) \|_{M_p^q} \leq C\,(1+t)^{-\delta_\low},\qquad t\geq0,
\end{equation}
where $\delta_\low$ is as in~\eqref{eq:deltalowtheta3}-\eqref{eq:deltaundamped}, namely,
\begin{equation}\label{eq:deltalow0}
\delta_\low = \min\left\{ \frac\kappa{a}, \ \frac\kappa{\tilde a} + \left(1-\frac{a}{\tilde a}\right)\frac{n-1}\tau d, \ \left(\frac{n-1}\tau +\frac1m\right)d \right\}.
\end{equation}
Estimate~\eqref{eq:estlowdisp} also holds if $\kappa=0$ (so that $\delta_\low=0$) and $1<p\leq 2\leq q<\infty$. If~$\kappa=0$ and $p=1$ or $q=\infty$, estimate~\eqref{eq:estlowdisp} also holds if we replace $L^1$ by the real Hardy space $H^1$ and $L^\infty$ by $\BMO$:
\begin{equation}\label{eq:estlowdispH1}
\begin{split}
&\| \rho^\mu\,e^{it\omega(\rho)}\chi_0(\rho) \|_{M(H^1,L^q)} \leq C,\qquad t\geq1, \ p=1, q\in[2,\infty),\\
&\| \rho^\mu\,e^{it\omega(\rho)}\chi_0(\rho) \|_{M(L^p,\BMO)} \leq C,\qquad t\geq1, \ p\in(1,2], q=\infty,\\
&\| \rho^\mu\,e^{it\omega(\rho)}\chi_0(\rho) \|_{M(H^1,\BMO)} \leq C,\qquad t\geq1, \ (p,q)=(1,\infty).
\end{split}
\end{equation}
If $w(\rho)\sim \rho^a$ as $\rho\to0$, then~\eqref{eq:estlowdisp} remains valid replacing $e^{it\omega}$ by $e^{it\omega}-e^{-it\omega}$, even if we weaken our assumption~$\kappa>0$ to $\kappa>-a$. When $\kappa=-a$, \eqref{eq:estlowdisp} remains valid if $1<p\leq 2\leq q<\infty$ and we get \eqref{eq:estlowdispH1} with $C$ replaced by $Ct$ in the right-hand side.
\end{theorem}
The proof of Theorem~\ref{thm:mainlowdisp} is similar to the proof of~\cite[Theorem 1.2]{CMY}, but we consider the endpoint estimates, namely, $p=1$ and/or $q=\infty$ in $L^p-L^q$ estimates. 

By duality, we may assume with no restriction that $p\leq q'$, since $M_{q'}^{p'}=M_p^q$ and $M(H^1,L^q)=M(L^{q'},\BMO)$. Since we are interested in $L^p-L^q$ estimates with $p\leq 2\leq q$, this means that $p\leq q'\leq2$, in view of the previous duality assumption. Therefore, we may interpolate the norm of the multiplier in the following spaces: $M_2^2=L^\infty$, $M_1^2=L^2$ and $M_1^\infty=\mathscr{F}(L^\infty)$. Since the first two estimates involve $L^\infty$ and $L^2$ estimate of the multiplier itself, the oscillations play no role and we immediately obtain that for any $N$:
\[
\|\rho^\mu\,\,e^{it\omega}\,\psi_N\|_{M_2^2} = \|\rho^\mu\,\,\psi_N\|_{L^\infty} \sim N^\mu,\qquad \|\rho^\mu\,\,e^{it\omega}\,\psi_N\|_{M_1^2} = \|\rho^\mu\,\,\psi_N\|_{L^2} \sim N^{\mu+\frac{n}2}.
\]
Interpolating the $M_2^2$ and the $M_1^2$ norms, we get
\begin{equation}\label{eq:p2}
\|\rho^\mu\,\,e^{it\omega}\,\psi_N\|_{M_{p_0}^2} \lesssim N^{n\left(\frac1{p_0}-\frac12\right)+\mu}, \quad p_0\in[1,2].
\end{equation}
The theory of oscillatory integrals (Lemmas~\ref{lem:low} and~\ref{lem:hi}) comes into play in the $M_1^\infty=\mathscr{F}(L^\infty)$ estimate of the multiplier, that is, we shall estimate the $L^\infty$ norm of $\mathscr{F}^{-1}(\rho^\mu\,e^{it\omega}\,\psi_N)$. We get
\[ \|\rho^\mu\,\,e^{it\omega}\,\psi_N\|_{M_1^\infty} = N^{n+\mu}\,\sup_{x\in\R^n} |I_N(x,t)|, \]
where the notation $I_N$ is as in~\eqref{eq:IN}. Interpolating the above estimate with~\eqref{eq:p2}, we obtain
%
%
\begin{equation}\label{eq:Mpq}
\| \rho^\mu\,\,e^{it\omega}\,\psi_N \|_{M_p^q}
    \lesssim N^{\left(n\left(\frac1{p_0}-\frac12\right)+\mu\right)(1-d)}\,\big(N^{n+\mu}\,\sup_{x\in\R^n} |I_N(x,t)|\big)^{d} \sim N^\kappa\,\big(\sup_{x\in\R^n} |I_N(x,t)|\big)^{d},
\end{equation}
where $d=1-2/q$ as in~\eqref{eq:d} (due to the duality assumption $p\leq q'$) and $\kappa$ is as in~\eqref{eq:kappa}. %
%
%
The following behavior of geometric sums will be used:
\begin{equation}\label{eq:sum}
\sum_{N_1}^{N_2} N^\gamma= \sum_{k=k_1}^{k_2} 2^{k\gamma} = 2^{k_1\gamma}\sum_{k=0}^{k_2-k_1} 2^{k\gamma}
\sim \begin{cases}
2^{k_2\gamma}\sim N_2^\gamma , & \gamma>0, \\
k_2-k_1=\log_2(N_2/N_1) , & \gamma=0,\\
2^{k_1\gamma} = N_1^\gamma , & \gamma<0,
\end{cases}
\end{equation}
where $N=2^k$, with $k\in\Z$, $N_j=2^{k_j}$ and $k_1\ll k_2$. The above asymptotic behavior also applies to the series with $N_1=0$ (that is, $k_1=-\infty$) if $\gamma>0$, and with $N_2=\infty$ (that is, $k_2=+\infty$) if $\gamma<0$.

\subsection{Proof of Theorem~\ref{thm:mainlowdisp}}

\begin{proof}
For the ease of reading, we assume $N_0=1$ with no loss of generality.

We first assume $d>0$, that is, $p<2<q$, and $\kappa>0$. For any fixed $t>0$, it is sufficient to estimate
\[ \sum_{N\leq1}\|\rho^\mu\,e^{it\omega}\,\psi_N\|_{M_p^q} \leq \sum_{N\leq1} N^\kappa <\infty, \]
due to~$\kappa>0$; in particular,
\[ \sum_{N\leq1} \|\rho^\mu\,e^{it\omega}\psi_N\|_{M_p^q} \leq C, \qquad t\geq0,\]
so we may assume $t\geq1$ in the following and prove the desired decay estimate with $1+t$ replaced by $t$, for simplicity.

Thanks to Lemma~\ref{lem:low}, for all $N\leq 1$, with the exception of a finite number of them, we may estimate
\[ \sup_{x\in\R^n} |I_N(x,t)| \lesssim (1+tN^a)^{-L}, \]
for any $L>0$; by using the interpolation in~\eqref{eq:Mpq}, we obtain
\begin{equation}\label{eq:Igood}
\|\rho^\mu\,e^{it\omega}\,\psi_N\|_{M_p^q} \lesssim N^\kappa\,(1+tN^a)^{-Ld}.
\end{equation}
Summing over those $N$,
\[ \sum \|\rho^\mu\,e^{it\omega}\psi_N\|_{M_p^q} \lesssim \sum_{N\leq \bar N} N^\kappa + \sum_{\bar N\leq N\leq 1} N^\kappa\,(tN^a)^{-Ld}, \]
where we fixed $\bar N\sim t^{-\frac1{a}}$, so that $t\bar N^a\sim 1$. By~\eqref{eq:sum}, we get
\[ \begin{split}
& \sum_{N\leq\bar N} N^\kappa \sim \bar N^{\kappa} \sim t^{-\frac{\kappa}{a}}; \quad \text{on the other hand,} \\
& \sum_{\bar N\leq N} N^\kappa\,(tN^a)^{-Ld} \sim \bar N^\kappa (t\bar N^a)^{-Ld} \sim \bar N^\kappa \sim t^{-\frac{\kappa}{a}},
\end{split}\]
for a sufficiently large $L$ (such that $\kappa-Lda<0$). We now consider the finite number of $N$ such that the estimate for $|I_N(x,t)|$ is given by~\eqref{eq:INlow}, that is,
\[ |I_N(x,t)| \lesssim 1 \wedge (tN^{a})^{-\frac{n-1}\tau}(1\wedge (tN^{\tilde a})^{-\frac1m}). \]
By the interpolation in~\eqref{eq:Mpq}, we obtain
\begin{equation}\label{eq:varphi}
\|\rho^\mu\,e^{it\omega}\,\psi_N\|_{M_p^q} \lesssim N^\kappa\,(1\wedge \varphi^d), \qquad \varphi(t,N) = (tN^{a})^{-\frac{n-1}\tau}(1\wedge (tN^{\tilde a})^{-\frac1m}).
\end{equation}
We claim that $1\wedge \varphi^d\sim 1$ if $N\leq \bar N$, while $1\wedge \varphi^d\sim\varphi^d$ if $N\geq\bar N$. Indeed, it is sufficient to notice that
\[ \begin{split}
& 1\lesssim (tN^{a})^{-\frac{n-1}\tau} \approx (tN^{a})^{-\frac{n-1}\tau}(1\wedge (tN^{\tilde a})^{-\frac1m}), \quad \text{if $N\leq \bar N$, while}\\
& (tN^{a})^{-\frac{n-1}\tau}(1\wedge (tN^{\tilde a})^{-\frac1m} ) \leq (tN^{a})^{-\frac{n-1}\tau} \lesssim 1, \quad \text{if $N\geq \bar N$.}
\end{split} \]
In the first estimate, we used that $\tilde a\geq a$.

Replacing $1\wedge \varphi^d\sim1$ for $N\leq \bar N$, we immediately obtain
\[ \sup_{N\leq \bar N}\,\|\rho^\mu\,e^{it\omega}\,\psi_N\|_{M_p^q} \lesssim \sup_{N\leq \bar N}\,N^\kappa = \bar N^\kappa, \]
as before, so we now consider $N\geq \bar N$, for which $1\wedge\varphi^d\sim\varphi^d$.

We distinguish two cases. First, let $\tilde a=a$. In this case, $\varphi(t,N) \sim (tN^{a})^{-\frac{n-1}\tau-\frac1m}$, so that
\[ \sup_{\bar N\leq N \leq 1} N^\kappa\,\varphi^d \sim \sup_{\bar N\leq N \leq 1} N^\kappa \, (tN^{a})^{-\left(\frac{n-1}\tau+\frac1m\right)d} \sim t^{-\min\left\{\frac{\kappa}{a},\left(\frac{n-1}\tau+\frac1m\right)d\right\}}, \]
the first term in the minimum being attained for $N\sim \bar N$, that is, if
\[ \frac{\kappa}a \leq \left(\frac{n-1}\tau+\frac1m\right)d,\]
and the second one being attained for $N=1$. This concludes the proof for $\kappa>0$ when $\tilde a=a$.

Let us consider the degenerate case, that is, $\tilde a>a$. We now get a ``loss'' of $N^{-\frac{\tilde a-a}m}$ in one of the two options of~\eqref{eq:varphi}. This loss suggests us to introduce a new index $\tilde N\sim t^{-\frac1{\tilde a}}$; we notice that $\bar N\ll \tilde N$, due to $a<\tilde a$ for $t\gg1$, and that
\[ 1\wedge (tN^{\tilde a})^{-\frac1m} \sim \begin{cases}
1 , & \bar N \leq N\leq \tilde N,\\
(tN^{\tilde a})^{-\frac1m} , & \tilde N\leq N\leq 1.
\end{cases} \]
Then
\[ \begin{split}
& \sup_{\bar N\leq N\leq \tilde N} N^{\kappa} (tN^{a})^{-\frac{n-1}\tau d} \sim t^{-\min\left\{\frac\kappa{a}, \frac\kappa{\tilde a}+\left(1-\frac{a}{\tilde a}\right)\frac{n-1}\tau d\right\}}, \\
& \sup_{\tilde N\leq N\leq 1} N^{\kappa} (tN^{a})^{-\frac{n-1}\tau d}\,(tN^{\tilde a})^{-\frac{d}m} \sim t^{-\min\left\{ \frac\kappa{\tilde a}+\left(1-\frac{a}{\tilde a}\right)\frac{n-1}\tau d,\left(\frac{n-1}\tau +\frac1m\right)d\right\}},
\end{split} \]
where the two terms in the minimum are attained at $N=\bar N$ and at $N=\tilde N$, respectively, in the first line, and at $N=\tilde N$ and at $N=1$, respectively, in the second line. Summarizing,
\[ \sup_{N\leq 1} \| \rho^\mu\,e^{it\omega}\,\psi_N \|_{M_p^q} \lesssim t^{-\min\left\{\frac\kappa{a}, \frac\kappa{\tilde a}+\left(1-\frac{a}{\tilde a}\right)\frac{n-1}\tau d,\left(\frac{n-1}\tau +\frac1m\right)d\right\}}, \]
where the middle term in the minimum is new with respect to the nondegenerate case.

This concludes the proof for $\kappa>0$.

Now let $\kappa=0$, that is, $\mu=-n(1/p-1/q)$. In this case, it is not necessary to rely on a dyadic partition for $N\leq\bar N$, neither it is useful since the oscillations do not provide any decay in this region. We put $\psi_\low = \sum_{N\leq \bar N} \psi_N$ and we apply Hardy-Littlewood-Sobolev Theorem~\ref{thm:HLS} in~\textsection\ref{sec:multipliers} (that is, $\rho^\mu=\rho^{-n\left(\frac1p-\frac1q\right)}\in M_p^q$) and Mikhlin-H\"ormander Theorem~\ref{thm:Mik} in~\textsection\ref{sec:multipliers} to obtain that
\begin{equation}\label{eq:MH0}
\begin{split}
& \|\rho^{-n\left(\frac1p-\frac1q\right)}\,e^{it\omega}\,\psi_\low \|_{M_p^q} \lesssim \|e^{it\omega}\,\psi_\low \|_{M_p^p} \leq C,\ \text{if $1<p\leq2\leq q<\infty$,}\\
& \|\rho^{-n\left(\frac1p-\frac1q\right)}\,e^{it\omega}\,\psi_\low \|_{M(H^1,L^q)} \lesssim \|e^{it\omega}\,\psi_\low \|_{M(H^1,H^1)} 
\leq C,\ \text{if $p=1$ and $q\in[2,\infty)$,}\\
& \|\rho^{-n\left(\frac1p-\frac1q\right)}\,e^{it\omega}\,\psi_\low \|_{M(H^1,\BMO)} \lesssim \|e^{it\omega}\,\psi_\low \|_{M(H^1,H^1)}\leq C,\ \text{if $(p,q)=(1,\infty)$.}
\end{split}
\end{equation}
Indeed, $|\partial_\xi^\gamma(e^{it\omega})|\lesssim \xii^{-|\gamma|}$, since $t|\omega^{(k)}|\lesssim tN^{a-k}\lesssim N^{-k}$ due to $N\leq \bar N$.

Finally, let us replace $e^{it\omega}$ by $e^{it\omega}-e^{-it\omega}$ as in~\eqref{eq:Phisinc}. Using $e^{it\omega}-e^{-it\omega}\sim 2t\omega\,e^{-it\omega}$ for $t|\omega|$ small, we still get
\[ \sum_{N\leq \bar N}\,\|\rho^\mu\,(e^{it\omega}-e^{-it\omega})\,\psi_N\|_{M_p^q} \lesssim t \sum_{N\leq \bar N}\,N^{\kappa+a} = t \,\bar N^{\kappa+a} \sim t^{-\frac\kappa{a}}, \]
provided that $\kappa>-a$, and similarly for $\sup_{N\leq \bar N}\,\|\rho^\mu\,(e^{it\omega}-e^{-it\omega})\,\psi_N\|_{M_p^q}$. In the case $\kappa=-a$, we may estimate as in~\eqref{eq:MH0}, replacing $L^1$ by $H^1$ and $L^\infty$ by $\BMO$. %
%
%
This concludes the proof for $d>0$. Finally, let $d=0$. In this case, when $\kappa>0$, by~\eqref{eq:sum}, we directly estimate
\[ \sum_{N\leq 1} \| \rho^\mu\,e^{it\omega}\psi_N \|_{M_p^q} \lesssim \sum_{N\leq1} N^\kappa \sim 1. \]
When $\kappa=0$ and $1<p\leq 2\leq q<\infty$, or when we consider $e^{it\omega}-e^{-it\omega}$, we modify the proof as in the case $d>0$.
\end{proof}


\subsection{Proof of Theorem~\ref{thm:mainlow}}

Here we prove Theorem~\ref{thm:mainlow} (in the general case, when $\alpha$ is not identically zero).
\begin{proof}
We may assume $t\geq1$ and prove the decay estimate with $1+t$ replaced by $t$. We first assume that $1<p\leq 2\leq q<\infty$. To take into account of the interplay with the dissipation in Theorem~\ref{thm:mainlow}, we plan to estimate
\[ \|\rho^\mu\,e^{it\omega-t\alpha}\,\chi_0\|_{M_p^q} \leq \|\rho^{\mu-\theta\Xi}\,e^{it\omega}\,\chi_0\|_{M_p^q}\, \|\rho^{\theta\Xi}\,e^{-t\alpha}\|_{M_q^q}, \]
for some $\Xi\geq0$. Using that $\alpha^\Xi\approx \rho^{\theta\Xi}$, we get (see, for instance, \cite[Lemmas 3.1, 3.2]{DAE25}):
\begin{equation}\label{eq:estdiss}
\|\rho^{\theta\Xi}\,e^{-t\alpha}\|_{M_q^q} \lesssim t^{-\Xi},
\end{equation}
for any $q\in [1, \infty]$. Then we apply Theorem~\ref{thm:mainlowdisp} with $\rho^\mu$ replaced by $\rho^{\mu-\theta\Xi}$ and $\kappa$ replaced by $\kappa-\Xi\theta$:
\begin{equation}\label{dispersive}
\begin{split}
& \|\rho^{\mu-\theta\Xi}\,e^{it\omega(\rho)}\,\chi_0(\rho)\|_{M_p^q}
    \lesssim t^{-\delta_\low(\Xi)}, \quad t\geq1,\\
& \delta_\low(\Xi)
    = \min\left\{\frac{\kappa-\Xi\theta}{a}, \frac{\kappa-\Xi\theta}{\tilde a}+\left(1-\frac{a}{\tilde a}\right)\frac{n-1}\tau d, \left(\frac{n-1}\tau +\frac1{m}\right)d \right\}.
\end{split}
\end{equation}
To maximize the decay rate $\Xi+\delta_\low(\Xi)$ with respect to $\Xi\in[0,\kappa/\theta]$, we compute
\[ \partial_\Xi (\Xi+ \delta_\low(\Xi)) = \begin{cases}
1-\frac\theta{a} , & \frac{\kappa-\Xi\theta}{a}<\frac{n-1}\tau d,\\
1-\frac\theta{\tilde a} , & 0<\frac{\kappa-\Xi\theta}{a}-\frac{n-1}\tau d<\frac1{m}\,\frac{\tilde a}{a} d,\\
1 , & \frac{\kappa-\Xi\theta}{a}>\left(\frac{n-1}\tau +\frac1{m}\,\frac{\tilde a}{a}\right) d.
\end{cases} \]
We distinguish five cases:
\begin{itemize}
\item If $\theta\leq a$, $\partial_\Xi (\Xi+ \delta_\low(\Xi))\geq 0$, so that the maximum is obtained taking the largest possible value for $\Xi$, that is, setting $\Xi\theta=\kappa$; hence, the decay is $t^{-\Xi}=t^{-\frac\kappa\theta}$.
\item If $a<\theta\leq\tilde a$, then $\partial_\Xi (\Xi+ \delta_\low(\Xi))<0$ in the first interval and $\partial_\Xi (\Xi+ \delta_\low(\Xi))\geq0$ in the second and in the third interval. As a consequence, the maximum is obtained taking $\Xi=0$ if $\frac\kappa{a}<\frac{n-1}\tau d$, and taking
\[ \Xi\theta=\kappa-\frac{n-1}\tau da,\]
otherwise; in this latter case, the decay is $t^{-\Xi-\frac{\kappa-\Xi\theta}a}=t^{-\frac\kappa\theta-\left(1-\frac{a}\theta\right)\frac{n-1}\tau d}$; that is,
\[ \delta_\low=\min\left\{\frac\kappa{a},\frac\kappa\theta+\left(1-\frac{a}\theta\right)\frac{n-1}\tau d\right\}. \]
\item If $\tilde a<\theta$, then $\partial_\Xi (\Xi+ \delta_\low(\Xi))<0$ in the first and in the second interval and $\partial_\Xi (\Xi+ \delta_\low(\Xi))>0$ in the third interval. As a consequence, the maximum is obtained taking $\Xi=0$ if $\frac\kappa{a}<\left(\frac{n-1}\tau +\frac1{m}\,\frac{\tilde a}{a}\right) d$, and taking
\[ \Xi\theta=\kappa-\left(\frac{n-1}\tau a+\frac1{m}\tilde a\right)d,\]
otherwise; in this latter case, the decay rate is
\[ t^{-\Xi-\frac{\kappa-\Xi\theta}{\tilde a}-\left(1-\frac{a}{\tilde a}\right)\frac{n-1}\tau d}=t^{-\frac\kappa\theta-\left(1-\frac{a}\theta\right)\frac{n-1}\tau d-\left(1-\frac{\tilde a}\theta\right)\frac1{m}d};\]
therefore,
\[ \delta_\low=\min\left\{\frac\kappa{a},\frac{\kappa}{\tilde a}+\left(1-\frac{a}{\tilde a}\right)\frac{n-1}\tau d,\frac\kappa\theta+\left(1-\frac{a}\theta\right)\frac{n-1}\tau d+\left(1-\frac{\tilde a}\theta\right)\frac1{m}d\right\}. \]
\end{itemize}
This leads to the decay rate in~\eqref{eq:estlow}.

Now we consider the case $p=1$ or $q=\infty$, under the assumption that $\kappa>0$. The proof follows as before, with one important exception. If $\theta\leq a$, or $\theta\leq \tilde{a}$ and $n=1$, then we set $\Xi$ such that $\kappa-\Xi\theta=0$.\footnote{The same consideration applies when $\theta>\tilde a$, $n=1$ and $m=\infty$, but this latter is a case with no dispersion ($\delta_\low=\kappa/\theta$), as for the one-dimensional strongly damped wave equation, so its interest is limited} Therefore, we shall replace the space $L^1$ by $H^1$ and the space $L^\infty$ by $\BMO$, as in~\eqref{eq:MH0}. Let us consider first the case $p=1$ and $q\in[2,\infty)$. In this case, we use~\cite[Corollary 3.1]{DAE25} to improve~\eqref{eq:estdiss} to an estimate from the space $L^1$ to its subspace $H^1$:
\begin{equation}\label{eq:estdissH1}
\|\rho^{\theta\Xi}\,\,e^{-t\alpha}\|_{M(L^1,H^1)} \lesssim t^{-\Xi},
\end{equation}
for $\Xi>0$, so that
\[ \|\rho^\mu\,e^{it\omega-t\alpha}\,\chi_0\|_{M_1^q} \leq \|\rho^{\theta\Xi}\,e^{-t\alpha}\|_{M(L^1,H^1)}\, \|\rho^{\mu-\theta\Xi}\,e^{it\omega}\,\chi_0\|_{M(H^1,L^q)}.  \]
By duality, we treat the case $1<p\leq 2$ and $q=\infty$, while for $(p,q)=(1,\infty)$, we estimate
\[ \begin{split}
& \|\rho^\mu\,e^{it\omega-t\alpha}\,\chi_0\|_{M_1^\infty}\\
    & \leq \|\rho^{\theta\frac\Xi2}\,e^{-t\frac{\alpha}2}\|_{M(L^1,H^1)}\, \|\rho^{\mu-\theta\Xi}\,\alpha^{-\Xi}\,e^{it\omega}\,\chi_0\|_{M(H^1,\BMO)}\,\|\rho^{\theta\frac\Xi2}\,e^{-t\frac{\alpha}2}\|_{M(\BMO,L^\infty)},
\end{split} \]
and we use the dual estimate $M(\BMO,L^\infty)$ of~\eqref{eq:estdissH1} as well.

The case when $e^{it\omega}$ is replaced by $e^{it\omega}-e^{-it\omega}$ is treated in a similar way. The only one case where the choice of $\Xi$ that maximizes the decay is different with respect of the previous scenario, occurs when $\theta<a$; in this case, we fix $\Xi\theta=\kappa+a$, so that the decay rate is $t^{1-\Xi}=t^{1-\frac{\kappa+a}\theta}$.
\end{proof}

\section{Proof of high frequencies estimates}\label{sec:theoremproofhi}

The proof of Theorem~\ref{thm:mainhi} is relatively similar to the proof of~\cite[Theorem 1.2]{CMY}, but we consider the endpoint estimates, namely, $p=1$ and/or $q=\infty$ in $L^p-L^q$ estimates.
\begin{proof}[Proof of Theorem~\ref{thm:mainhi}]
For the ease of reading, we assume $N_1=1$ with no loss of generality.

First let $d>0$, that is, $p<2<q$.

Thanks to Lemma~\ref{lem:hi}, for all $N\geq 1$, with the exception of a finite number of them, we get~\eqref{eq:Igood}, as in the proof of Theorem~\ref{thm:mainlow}. For $t\in(0,1]$ we fix $\bar N$ such that
\[ t \bar N^b\sim 1,\qquad \text{i.e.}, \quad \bar N \sim t^{-\frac1{b}}, \]
similarly to what we did in the proof of Theorem~\ref{thm:mainlow}. For $t\geq1$, we fix $\bar N=1$. Summing over the $N$ for which~\eqref{eq:Igood} holds (the first sum in the right-hand side does not appear if $t\geq1$), by~\eqref{eq:sum} we get
\[
\sum \|\rho^\mu\,e^{it\omega}\,\psi_N\|_{M_p^q} \lesssim \sum_{1\leq N\leq \bar N} N^\kappa + \sum_{\bar N\leq N} N^\kappa\,(tN^b)^{-Ld} \sim \begin{cases}
t^{-\frac{\kappa}{b}} , & t\in(0,1],\\
t^{-Ld} , & t\geq1,
\end{cases} \]
%
%
for any sufficiently large $L$ (such that $\kappa-Ldb<0$).

We now consider the finite number of $N$ such that the estimate for $|I_N(x,t)|$ is given by~\eqref{eq:INhi}:
\[ |I_N(x,t)| \lesssim 1 \wedge (tN^{b})^{-\frac{n-1}\tau}(1\wedge (tN^{\tilde b})^{-\frac1m}). \]
By the interpolation in~\eqref{eq:Mpq}, we obtain
\begin{equation}\label{eq:varphihi}
\|\rho^\mu\,e^{it\omega}\,\psi_N\|_{M_p^q} \lesssim N^\kappa\,(1\wedge \varphi^d), \qquad \varphi(t,N) = (tN^{b})^{-\frac{n-1}\tau}(1\wedge (tN^{\tilde b})^{-\frac1m}).
\end{equation}
Indeed,
\[ \begin{split}
& 1\lesssim (tN^{b})^{-\frac{n-1}\tau}\approx (tN^{b})^{-\frac{n-1}\tau}(1\wedge (tN^{\tilde b})^{-\frac1m})\quad \text{if $N\leq \bar N$, while}\\
& (tN^{b})^{-\frac{n-1}\tau}(1\wedge (tN^{\tilde b})^{-\frac1m}) \leq (tN^{b})^{-\frac{n-1}\tau} \lesssim1,\quad \text{if $N\geq \bar N$.}
\end{split}\]
In the first estimate, we used that $\tilde b\leq b$.

If $t\in(0,1]$ and $N\leq \bar N$, by $1\wedge \varphi^d\sim1$, we immediately obtain $\|\rho^\mu\,e^{it\omega}\,\psi_N\|_{M_p^q} \lesssim N^\kappa \leq \bar N^\kappa$, as before, so we now consider $N\geq \bar N$, for which $1\wedge\varphi^d\sim\varphi^d$.

We distinguish two cases. First, let $\tilde b=b$. In this case, $\varphi(t,N) = (tN^{b})^{-\frac{n-1}\tau-\frac1m}$, so that
\[ \sup_{\bar N\leq N} N^\kappa\,\varphi^d \sim \sup_{\bar N\leq N} N^\kappa \, (tN^{b})^{-\left(\frac{n-1}\tau+\frac1m\right)d} \sim \begin{cases}
t^{-\left(\frac{n-1}\tau+\frac1m\right)d} , & t\geq1,\\
t^{-\frac{\kappa}{b}} , & t\in(0,1],
\end{cases} \]
under the assumption~\eqref{eq:reg} (the supremum is attained at $N\sim\bar N$).

Now let us consider the degenerate case, that is, $\tilde b<0<b$.

First, let $t\geq1$. Then we set $\tilde N\sim t^{\frac1{-\tilde b}}$; due to
\[ 1\wedge (tN^{\tilde b})^{-\frac1m} \sim \begin{cases}
(tN^{\tilde b})^{-\frac1m} , & 1 \leq N\leq \tilde N,\\
1 , & \tilde N\leq N,
\end{cases} \]
we get
\[ \sup_{1\leq N\leq \tilde N} N^{\kappa} (tN^{b})^{-\frac{n-1}\tau d}\,(tN^{\tilde b})^{-\frac{d}m} \sim t^{-\min\left\{\left(\frac{n-1}\tau + \frac1m\right)d, \frac\kappa{\tilde b} + \left(1-\frac{b}{\tilde b}\right)\frac{n-1}\tau d \right\}}, \]
where the two terms in the minimum are attained at $N=1$ and at $N=\tilde N$, respectively. Similarly,
\[ \sup_{\tilde N\leq N} N^{\kappa} (tN^{b})^{-\frac{n-1}\tau d} \sim t^{\frac\kappa{-\tilde b}-\left(1+\frac{b}{-\tilde b}\right)\frac{n-1}\tau d}, \]
provided that~\eqref{eq:shrink} holds. Now, let $t\in(0,1]$. In this case, $1\wedge (tN^{\tilde b})^{-\frac1m}\sim 1$, so that
\[ \sup_{\bar N\leq N} N^\kappa\,(tN^{b})^{-\frac{n-1}\tau d} \sim \bar N^\kappa \sim t^{-\frac\kappa{b}}, \]
provided that~\eqref{eq:shrink} holds. This concludes the proof for $\kappa>0$.

Now, let $1<p\leq 2\leq q<\infty$ and $\kappa=0$. In this case, as in the end of the proof of Theorem~\ref{thm:mainlow}, it is not necessary to rely on a dyadic partition for $t\in(0,1]$ and $1\leq N\leq\bar N$, neither it is useful since the oscillations do not provide any decay in this region. We put $\psi_\low = \sum_{1\leq N\leq \bar N} \psi_N$ and we get~\eqref{eq:MH0} again.

In the case $\kappa<0$, by~\eqref{eq:sum} we may just replace $\sum_{1\leq N\leq\bar N} N^\kappa \sim 1$.


When $d=0$, the proof given for $d>0$ remains valid for $t\in(0,1]$, while for $t\geq1$, we proceed as in the end of the proof of Theorem~\ref{thm:mainlowdisp}.
\end{proof}
\begin{remark}\label{rem:kappaneghi}
The case $\kappa=0$ may also be included in Theorem~\ref{thm:mainhi} when $p=1$ or $q=\infty$. When $t\ll1$ and $p=1$, recalling that $M_1^q=\mathscr{F}(L^q)\subset L^{q'}$, it is sufficient to replace~\eqref{eq:MH0} by
\[
\|\rho^\mu\,e^{it\omega}\,\psi_\low \|_{M_1^q} = \|\mathscr{F}^{-1}(\rho^\mu\,e^{it\omega}\,\psi_\low)\|_{L^q} \lesssim \|\rho^\mu\,e^{it\omega}\,\psi_\low \|_{L^{q'}}\lesssim (-\log t)^{1-\frac1q},
\]
where we used that $\bar N\sim -\log_2t\sim -\log t$. By duality ($M_p^\infty=M_1^{p'}$) we may also deal with $q=\infty$. 
\end{remark}


We now prove~Theorem~\ref{thm:mainhiXi}.
%
%
\begin{proof}[Proof of~Theorem~\ref{thm:mainhiXi}]
First let $1<p\leq 2\leq q<\infty$. The presence of a dissipation generates an exponential decay in time at high frequencies, in the following sense:
\begin{equation}\label{eq:exponential}
\| \alpha^{\Xi}\,e^{-t\alpha(\rho^2)}\chi_1\|_{M_q^q} \lesssim e^{-ct}, \quad t\geq1.
\end{equation}
On the other hand, for short time, $t\in(0,1]$, we use once again~\eqref{eq:estdiss}. %
%
%
Applying Theorem \ref{thm:mainhi}, we get
\[
\|\rho^{\mu-\theta\Xi}\,e^{it\omega(\rho)}\,\chi_1(\rho)\|_{M_q^q}\lesssim t^{-\frac{\kappa -\Xi\theta}{b}},
\]
for $t\in(0,1]$ and some polynomial decay for $t\geq1$, provided  that
\[\frac{\kappa-\Xi\theta}{b} \leq \begin{cases}
\left(\frac{n-1}\tau + \frac1m\right) d , & \tilde b=b,\\
\frac{n-1}\tau  d , & \tilde b<0<b.
\end{cases}\]
The polynomial decay for $t\geq1$ is dominated by the exponential decay in~\eqref{eq:exponential}, while the estimate for $t\in(0,1]$ is given by $t^{-\Xi-\frac{\kappa -\Xi\theta}{b}}$ as in the proof of Theorem~\ref{thm:mainlow}. The proof is concluded by choosing $\Xi$ as in \eqref{eq:Xi}.

In the case $p=1$ or $q=\infty$, we proceed as in the proof of Theorem~\ref{thm:mainlow}, replacing $L^1$ by $H^1$ and $L^\infty$ by $\BMO$ in Theorem~\ref{thm:mainhi} when $\kappa=0$, and relying on~\eqref{eq:estdissH1} in place of~\eqref{eq:estdiss}.
\end{proof}


\appendix

\section{The improved Boussinesq equation: a case of regularity loss}\label{sec:bneg}

Here we discuss the case when $\rho\omega'(\rho)\to0$ as $\rho\to\infty$, which occurs, for instance, in the Improved Boussinesq model:
\begin{equation}\label{eq:VIBQ}\tag{IBQ-eq}
u_{tt} + Au_{tt} + A u + 2\e A u_t =0.
\end{equation}
After rewriting~\eqref{eq:VIBQ} in normal form,
\begin{equation}\label{eq:VIBQnormal}
u_{tt} + (\Id+A)^{-1}\,A u + 2\e (\Id+A)^{-1}\,A u_t =0,
\end{equation}
we see that~\eqref{eq:VIBQnormal} is in the form of \eqref{eq:abgeneral}.

%
%
%
%
We consider the initial value problem for the Improved Boussinesq equation~\eqref{eq:VIBQ}:
\begin{equation}\label{eq:IBQ}\tag{IBQ}
\begin{cases}
u_{tt} +A u_{tt} + A u +2\e Au_t=0 , & t>0,\ x\in\R^n,\\
(u,u_t)(0,x)=(u_0,u_1)(x).
\end{cases}
\end{equation}
With the notation in~\eqref{eq:abgeneral}, $\alpha=\e \rho^2\<\rho\>^{-2}$ and $\beta=\rho^2\<\rho\>^{-2}$, so that
\[ \omega = \rho\<\rho\>^{-2}\,\sqrt{1+\rho^2(1-\e^2)}. \]
By straightforward computation:
\begin{equation}\label{eq:derVIBQ}
\begin{split}
\omega'(\rho)
    & = \frac{1+\rho^2(1-2\e^2)}{\sqrt{1+\rho^2(1-\e^2)}}\,\<\rho\>^{-4}>0,\\
\omega''(\rho)
    & = -\rho\,\frac{3\rho^4(1-3\e^2+2\e^4)+2\rho^2(3-3\e^2-\e^4)+3(1+\e^2)}{(\<\rho\>^8-\e^2\rho^2\<\rho\>^6)\sqrt{1+\rho^2(1-\e^2)}}<0,
\end{split}
\end{equation}
provided that $2\e^2<1$, for any~$\rho>0$. The case $\e=0$ corresponds to the inviscid Improved Boussinesq model. With the same notation in Theorems~\ref{thm:mainlow} and~\ref{thm:mainhi}, we have $a=1$, $\tilde a=3$, $b=\tilde b=-2$. In particular, $\omega''$ has the degeneracy at low frequencies of the Boussinesq equation~\eqref{eq:BQ}. Let $\ell\in\N$ and $\nu,\nu_0\in\R$. Since $|i\omega+\alpha|^\ell \approx (\rho\<\rho\>^{-1})^{\ell}$, we define $\kappa_\low=n(1/p-1/q)+\nu+\ell$ as in~\eqref{eq:klowBQ} , and
\begin{equation}
\label{eq:khiIBQ}
\kappa_\high = n\left(\frac1p-\frac1q\right)+\nu-\nu_0.
\end{equation}
%
%
%
%
The solution to~\eqref{eq:IBQ} with $\e=0$ verifies the low frequencies estimate~\eqref{eq:BQlow}, with $\kappa_\low$ as in~\eqref{eq:klowBQ} and $\delta_\BQ$ as in~\eqref{eq:deltaBQ}. In the viscous case, $\delta_\BQ$ is improved to $\delta_\VBQ$ as in~\eqref{eq:deltaVBQ}, due to $\alpha=\e\rho^2\<\rho\>^{-2}\approx \rho^2$ at low frequencies.

To deal with high frequencies, however, we need a result analogous to Theorem~\ref{thm:mainhi} with $b<0$.
\begin{theorem}\label{thm:hineg}
For any~$\xi$ with $\rho(\xi)\geq N_1/2$, we assume the following:
\begin{itemize}
\item that~\eqref{eq:wprimehi} and~\eqref{eq:wboundhi} hold for some $b<0$;
\item that~\eqref{eq:derhi} holds for some $m\geq2$ and $\tilde b=b$.
\end{itemize}
Let $1\leq p\leq 2\leq q\leq\infty$, $d=d(p,q)$ as in~\eqref{eq:d} and~$\kappa$ as in~\eqref{eq:kappa}. Then
\begin{equation}\label{eq:esthineg}
\| \rho^\mu\,e^{it\omega(\rho)}\chi_1(\rho) \|_{M_p^q} \leq C\,(1+t)^{-\min\left\{\left(\frac{n-1}\tau+\frac1m\right)d,\frac\kappa{b}\right\}},\qquad t\geq0,
\end{equation}
provided that $\mu\in\R$ is such that $\kappa<0$, or $\kappa\leq0$ if $1<p\leq 2\leq q<\infty$. If we consider $e^{it\omega}-e^{-it\omega}$ as in~\eqref{eq:Phisinc} and we assume that $w\sim \rho^{b}$ as $\rho\to\infty$, then~\eqref{eq:esthineg} remains valid even if we weaken our assumption to $\kappa<-b$, or $\kappa\leq-b$, if $1<p\leq 2\leq q<\infty$.
\end{theorem}
\begin{proof}[Proof of Theorem~\ref{thm:hineg}]
We assume $d>0$, since the case $d=0$ may be treated as in the end of the proofs of Theorems~\ref{thm:mainlowdisp} and~\ref{thm:mainhi}. We proceed as in the proof of Theorem~\ref{thm:mainhi}, but we now define $\bar N\sim t^{\frac1{-b}}$ for $t\geq1$, and $\bar N=1$ for $t\in(0,1]$, so that (the first term in the right-hand side does not appear for $t\in(0,1]$), by~\eqref{eq:sum} we get
\[ \sum \|\rho^\mu\,e^{it\omega}\,\psi_N\|_{M_p^q} \lesssim \sum_{1\leq N\leq \bar N} N^\kappa\,(tN^{b})^{-Ld} + \sum_{\bar N\leq N} N^\kappa \sim \begin{cases}
1 , & t\in(0,1], \\
t^{-\frac{\kappa}{b}} , & t\geq1,
\end{cases} \]
for any sufficiently large $L$ (such that $\kappa+Ld(-b)>0$), provided that $\kappa<0$ (no regularizing effect is obtained by the dispersion). We now consider the finite number of $N$ such that the estimate for $|I_N(x,t)|$ is given by~\eqref{eq:INhi}. Then~\eqref{eq:varphihi} holds. %
%
%
If $N\geq \bar N$, by $1\wedge \varphi^d\sim1$, we immediately obtain
\[ \|\rho^\mu\,e^{it\omega}\,\psi_N\|_{M_p^q} \lesssim N^\kappa \leq \bar N^\kappa, \]
as before, so we now consider $t\geq1$ and $N\leq \bar N$, for which $1\wedge\varphi^d\sim\varphi^d$. Then, 
\[ \sup_{1\leq N \leq \bar N} N^\kappa\,\varphi^d \sim \sup_{1\leq N \leq \bar N} N^\kappa \, (tN^{b})^{-\left(\frac{n-1}\tau+\frac1m\right)d} \sim t^{-\min\left\{\frac\kappa{b},\left(\frac{n-1}\tau+\frac1m\right)d\right\}}.\]
%
This concludes the proof for $\kappa<0$. The case $\kappa=0$ when $1<p\leq 2\leq q<\infty$, and the case of $e^{it\omega}-e^{-it\omega}$ as in~\eqref{eq:Phisinc},
may be treated as in the end of the proof of Theorem~\ref{thm:mainlowdisp}.
\end{proof}
Applying Theorem~\ref{thm:hineg} with $b=-2$, the solution to~\eqref{eq:IBQ} with $\e=0$ verifies the high frequencies estimate
\begin{equation}\label{eq:IBQhi}
\||D|^\nu \partial_t^\ell \mathscr{F}^{-1}(\chi_1\hat u)(t,\cdot) \|_{L^q} \lesssim (1+t)^{-\min\left\{\left(\frac{n-1}\tau+\frac12\right)d,\frac{-\kappa_\high}2\right\}}\, \|\<D\>^{\nu_0} (u_0,u_1)\|_{L^p}, \qquad t\geq0.
\end{equation}
In particular, the solution to~\eqref{eq:IBQ} with $\e=0$ verifies the long time estimate
\begin{equation}\label{eq:estIBQ}\begin{split}
& \||D|^\nu \partial_t^\ell  u (t,\cdot) \|_{L^q} \lesssim t^{-\delta_\IBQ}\, \|\<D\>^{\nu_0} (u_0,|D|^{-1}\<D\>u_1)\|_{L^p}, \qquad t\geq1,\\
& \delta_\IBQ=\min \left\{\delta_\BQ, \frac{-\kappa_\high}2\right\}.
\end{split}\end{equation}
We stress that the decay rate structure in~\eqref{eq:estIBQ} shows the ``regularity-loss decay'' property: enlarging the data regularity $\nu_0$ may increase the decay rate, since $-\kappa_\high$ is increasing with respect to $\nu_0$.

Assuming $q=p'$, so that $\kappa_\high=nd+\nu-\nu_0$, and
\[ \kappa_\low\geq \left(\frac{n-1}\tau +\frac32\right) d, \quad \text{so that}\qquad \delta_\BQ = \left(\frac{n-1}\tau+\frac12\right)d, \]
we find that $\delta_\IBQ=\delta_\BQ$ if, and only if,
\[ \nu_0 \geq \nu + 2\left(\frac{n-1}\tau+\frac12\right)d + nd. \]
In the case $\tau=2$ and $\nu=n/q$, we find the same regularity $\nu_0\geq (2-3/q)n$ obtained in \cite[Theorem~1]{Cho-Ozawa=2007} for $A=-\Delta$.

When $\e\in(0,\e_0]$, the dissipation produces an exponential decay, but it does not produce smoothing as it happened for the Boussinesq equation, due to~$\alpha\approx 1$ at high frequencies, that is, $\theta=0$. Therefore, the restriction from above on $\kappa_\high$ remains for the viscous model. However, the decay rate in the estimate is independent on $\kappa_\high$, in the sense that the solution to~\eqref{eq:IBQ} with $\e\in(0,\e_0]$ verifies the high frequencies estimate
\begin{equation}\label{eq:estVIBQ}
\||D|^\nu \partial_t^\ell  u (t,\cdot) \|_{L^q} \leq C (1+t)^{-\delta_\VBQ}\, \|\<D\>^{\nu_0} (u_0,|D|^{-1}\<D\>u_1)\|_{L^p}, \qquad t\geq0,
\end{equation}
provided that $\nu_0$ is sufficiently large that $\kappa_\high$ verifies the desired bound from above.
\begin{remark}\label{rem:IBQ1inf}
Let $A=-\Delta $ ($\tau=2$) and $\e=0$. Estimates for the Improved Boussinesq equation are derived  in \cite{Cho-Ozawa=2006} for $n=1$ and in \cite{Cho-Ozawa=2007} for $n\geq 2$.
In Lemmas 2.2 and 2.3 in \cite{Cho-Ozawa=2006}, in space dimension $n=1$, the authors obtained the estimate
\begin{equation}\label{eq:OC}
\| u (t,\cdot) \|_{L^\infty} \lesssim (1+t)^{-\frac13}\,\left( \| (u_0,u_1)\|_{L^1} + \| |D|^{-1}u_1\|_{L^1}+
 \| |D|^s(u_0,u_1)\|_{L^{6/5}}\right),
\end{equation}
under the assumption $s>3/2$. To compare the data regularity of our result with the data regularity in~\eqref{eq:OC}, we shall mix $L^1-L^\infty$ low frequencies estimates and $L^p-L^\infty$ high frequencies estimates. Setting $n=1$, $q=\infty$ and $\ell=\nu=0$ in~\eqref{eq:IBQhi}, for any $p\in[1,2]$ we find
\[ \|\mathscr{F}^{-1}(\chi_1\hat u)(t,\cdot) \|_{L^\infty} \lesssim (1+t)^{-\min\left\{\frac1p-\frac12,\frac{\nu_0}2-\frac1{2p}\right\}}\, \|\<D\>^{\nu_0} (u_0,u_1)\|_{L^p}, \qquad t\geq0, \]
in particular, we see that $p=6/5$ is the largest possible choice which guarantees that $1/p-1/2\geq 1/3$, the desired decay rate. Once we fixed $p=6/5$, we find $\nu_0=3/2$, obtaining
\[ \|u (t,\cdot) \|_{L^q} \lesssim (1+t)^{-\frac13}\,\big(\|(u_0,|D|^{-1}u_1)\|_{L^1} + \|\<D\>^{\frac32} (u_0,u_1)\|_{L^{\frac65}} \big), \qquad t\geq0.\]
We stress that with respect to the estimate in~\eqref{eq:OC}, it is sufficient for us to assume the critical regularity~$\nu_0=3/2$ in $L^{\frac65}$ and we do not need to assume~$\nu_0>3/2$. Also, we do not need to assume $u_1\in L^1$.

We incidentally notice that the choice of $p=6/5$ in~\cite{Cho-Ozawa=2006} is optimal, in the sense that to use $L^1-L^\infty$ high frequencies estimates, obtaining the desired decay rate $t^{-\frac13}$, we shall assume $\<D\>^{\frac53}(u_0,u_1)\in L^1$, but by the Sobolev embeddings,
\[ \|\<D\>^{\frac32} (u_0,u_1)\|_{L^{\frac65}} \lesssim \|\<D\>^{\frac53}(u_0,u_1)\|_{L^1}. \]
\end{remark}

\section{Degeneracy at intermediate points}\label{sec:inter}

Here we briefly discuss the case whether $\omega''$ vanishes at some given $\bar \rho>0$ with order $r-2$, that is, $\omega^{(k)}(\bar \rho)=0$ for $k=2,\ldots,r-1$. If $r\leq m$, where $m$ is as in Theorem~\ref{thm:mainlow} or as in Theorem~\ref{thm:mainhi}, it can be included in that statement with no modification. However, an example for which the influence from this degeneracy is relevant is provided by the plate equation with rotational inertia:
\begin{equation}\tag{PR}\label{eq:rot}
\begin{cases}
u_{tt} +A u_{tt} + A^2 u=0 , & t>0,\ x\in\R^n,\\
(u,u_t)(0,x)=(u_0,u_1)(x).
\end{cases}
\end{equation}
Then,
\[ \omega(\rho)=\rho^2\<\rho\>^{-1},\qquad
\omega'
    = 
    \frac{2\rho+\rho^3}{\<\rho\>^3}>0,\qquad
\omega''
    = \frac{2-\rho^2}{\<\rho\>^5}\,.
\]
In particular, $\omega''(\sqrt{2})=0$, so we need to rely on $\omega'''$ near $\rho=\sqrt{2}$. We find that %
%
%
$\omega'''(\sqrt{2})<0$. %

As a consequence, it is possible to apply Theorem~\ref{thm:mainlowdisp} with $N_0=1/2$, and Theorem~\ref{thm:mainhi} with $N_1=4$, but we need a different tool to deal with $N=1$ and $N=2$.
\begin{proposition}\label{prop:maininter}
For any~$\xi$ with $N_0/2\leq\rho(\xi)\leq 2N_1$, we assume that~$\omega$ and $g$ are smooth and that~$\sum_{k=2}^m |\omega^{k}(\rho)|\neq0$ for some $m\geq2$. Let $1\leq p\leq 2\leq q\leq\infty$ and $d=d(p,q)$ as in~\eqref{eq:d}.
Then,
\begin{equation}\label{eq:estinter}
\| \rho^\mu e^{it\omega(\rho)}(1-\chi_0(\rho)-\chi_1(\rho)) \|_{M_p^q} \leq C\,(1+t)^{-\left(\frac{n-1}\tau+\frac1m\right)d}, \quad t\geq0.
\end{equation}
\end{proposition}
\begin{proof}
The analogous result to Lemmas~\ref{lem:low} and~\ref{lem:hi} easily follows, since $N$ ranges in a finite set $N=N_0,\ldots,N_1$. The only difference is the use of van der Corput lemma (see~\cite[Chapter VIII, Proposition 2]{SteinHarmonic}) of order $k+2$ at each stationary point, where $k$ is the order of zero of $\omega''$; here $k=0$ if $\omega''(\rho)\neq0$ at a stationary point.
\end{proof}
%
%
%
%
%
We fix $N_0=1/2$ and $N_1=4$. Using the notation in Theorems~\ref{thm:mainlow} and~\ref{thm:mainhi}, we have $a=\tilde a=2$, $b=1$, $\tilde b=-1$, so $\omega''$ is degenerate at intermediate and high frequencies. Let $\ell\in\N$ and $\nu,\nu_0\in\R$. We define
\[ \begin{split}
\kappa_\low
    & = n(1/p-1/q)+\nu+2\ell,\\
\kappa_\high
    & = n(1/p-1/q)+\nu-\nu_0+\ell.
\end{split} \]
Applying Theorem~\ref{thm:mainlow} with $\kappa=\kappa_\low$, we find that the solution to~\eqref{eq:rot} verifies the low-frequencies estimate with $\delta_\low$ as in~\eqref{eq:deltalowtheta3}-\eqref{eq:deltaundamped}. Applying also Proposition~\ref{prop:maininter} with $m=3$, we get
\[ \||D|^\nu \partial_t^\ell \mathscr{F}^{-1}((1-\chi_1)\hat u)(t,\cdot) \|_{L^q} \lesssim (1+t)^{-\min\left\{\frac{\kappa_\low}2,\left(\frac{n-1}\tau+\frac13\right)d\right\}}\,\|(u_0,|D|^{-2}u_1)\|_{L^p}. \]
The structure of this estimate is the same of the estimate for the plate equation in~\textsection\ref{sec:plate}, exception given for the fact that $w''$ vanishes at $\rho=\sqrt{2}$, so that $d/2$ is replaced by $d/3$.

We assume that $\kappa_\high\leq \frac{n-1}\tau d$ and that $\kappa_\high>0$, or $\kappa_\high\geq0$ if $1<p\leq 2\leq q<\infty$. If $u_0=0$, we may relax the assumption to $\kappa_\high>-1$, or $\kappa_\high\geq -1$ if $1<p\leq 2\leq q<\infty$. Applying Theorem~\ref{thm:mainhi}, we get that the solution to~\eqref{eq:rot} verifies the high frequencies estimate~\eqref{eq:KGhi} with $|D|^{-2}\<D\>^{\nu_0+1}u_1$ in place of $\<D\>^{\nu_0-1}u_1$. 
%
%
In particular, the solution to~\eqref{eq:rot} verifies the long time estimate
\begin{equation}\label{eq:rotest}\begin{split}
& \qquad \||D|^\nu \partial_t^\ell u(t,\cdot) \|_{L^q} \lesssim t^{-\delta_\ROT} \|(\<D\>^{\nu_0} u_0,|D|^{-2}\<D\>^{{\nu_0}+1}u_1)\|_{L^p},\qquad t\geq1,\\
& \delta_\ROT = \min\left\{\frac{n}2\left(\frac1p-\frac1q\right)+\frac{\nu}2+\ell, \left(\frac{n-1}\tau+\frac13\right)d, 2\frac{n-1}\tau d-\kappa_\high\right\}.
\end{split}
\end{equation}

\end{document}